\documentclass[11pt, reqno]{amsart}
\usepackage{amssymb}
\usepackage{eucal}
\usepackage{amsmath}
\usepackage{amscd}
\usepackage[dvips]{color}
\usepackage{multicol}
\usepackage[all]{xy}           
\usepackage{graphicx}
\usepackage{color}
\usepackage{colordvi}
\usepackage{xspace}
\usepackage{tikz}

\usepackage[active]{srcltx} 

\begin{document}
\newcommand {\emptycomment}[1]{} 

\baselineskip=14pt
\newcommand{\nc}{\newcommand}
\newcommand{\delete}[1]{}
\nc{\mfootnote}[1]{\footnote{#1}} 
\nc{\todo}[1]{\tred{To do:} #1}

\nc{\mlabel}[1]{\label{#1}}  
\nc{\mcite}[1]{\cite{#1}}  
\nc{\mref}[1]{\ref{#1}}  
\nc{\mbibitem}[1]{\bibitem{#1}} 

\delete{
\nc{\mlabel}[1]{\label{#1}  
{\hfill \hspace{1cm}{\bf{{\ }\hfill(#1)}}}}
\nc{\mcite}[1]{\cite{#1}{{\bf{{\ }(#1)}}}}  
\nc{\mref}[1]{\ref{#1}{{\bf{{\ }(#1)}}}}  
\nc{\mbibitem}[1]{\bibitem[\bf #1]{#1}} 
}

\newtheorem{thm}{Theorem}[section]
\newtheorem{lem}[thm]{Lemma}
\newtheorem{cor}[thm]{Corollary}
\newtheorem{pro}[thm]{Proposition}
\newtheorem{ex}[thm]{Example}
\newtheorem{rmk}[thm]{Remark}
\newtheorem{defi}[thm]{Definition}
\newtheorem{pdef}[thm]{Proposition-Definition}
\newtheorem{condition}[thm]{Condition}

\renewcommand{\labelenumi}{{\rm(\alph{enumi})}}
\renewcommand{\theenumi}{\alph{enumi}}

\nc{\tred}[1]{\textcolor{red}{#1}}
\nc{\tblue}[1]{\textcolor{blue}{#1}}
\nc{\tgreen}[1]{\textcolor{green}{#1}}
\nc{\tpurple}[1]{\textcolor{purple}{#1}}
\nc{\btred}[1]{\textcolor{red}{\bf #1}}
\nc{\btblue}[1]{\textcolor{blue}{\bf #1}}
\nc{\btgreen}[1]{\textcolor{green}{\bf #1}}
\nc{\btpurple}[1]{\textcolor{purple}{\bf #1}}

\nc{\cm}[1]{\textcolor{red}{Chengming:#1}}
\nc{\yy}[1]{\textcolor{blue}{Yanyong: #1}}
\nc{\lit}[2]{\textcolor{blue}{#1}{}} 
\nc{\yh}[1]{\textcolor{green}{Yunhe: #1}}


\nc{\twovec}[2]{\left(\begin{array}{c} #1 \\ #2\end{array} \right )}
\nc{\threevec}[3]{\left(\begin{array}{c} #1 \\ #2 \\ #3 \end{array}\right )}
\nc{\twomatrix}[4]{\left(\begin{array}{cc} #1 & #2\\ #3 & #4 \end{array} \right)}
\nc{\threematrix}[9]{{\left(\begin{matrix} #1 & #2 & #3\\ #4 & #5 & #6 \\ #7 & #8 & #9 \end{matrix} \right)}}
\nc{\twodet}[4]{\left|\begin{array}{cc} #1 & #2\\ #3 & #4 \end{array} \right|}

\nc{\rk}{\mathrm{r}}
\newcommand{\g}{\mathfrak g}
\newcommand{\h}{\mathfrak h}
\newcommand{\pf}{\noindent{$Proof$.}\ }
\newcommand{\frkg}{\mathfrak g}
\newcommand{\frkh}{\mathfrak h}
\newcommand{\Id}{\rm{Id}}
\newcommand{\gl}{\mathfrak {gl}}
\newcommand{\ad}{\mathrm{ad}}
\newcommand{\add}{\frka\frkd}
\newcommand{\frka}{\mathfrak a}
\newcommand{\frkb}{\mathfrak b}
\newcommand{\frkc}{\mathfrak c}
\newcommand{\frkd}{\mathfrak d}
\newcommand {\comment}[1]{{\marginpar{*}\scriptsize\textbf{Comments:} #1}}

\nc{\gensp}{V} 
\nc{\relsp}{\Lambda} 
\nc{\leafsp}{X}    
\nc{\treesp}{\overline{\calt}} 

\nc{\vin}{{\mathrm Vin}}    
\nc{\lin}{{\mathrm Lin}}    

\nc{\gop}{{\,\omega\,}}     
\nc{\gopb}{{\,\nu\,}}
\nc{\svec}[2]{{\tiny\left(\begin{matrix}#1\\
#2\end{matrix}\right)\,}}  
\nc{\ssvec}[2]{{\tiny\left(\begin{matrix}#1\\
#2\end{matrix}\right)\,}} 

\nc{\typeI}{local cocycle $3$-Lie bialgebra\xspace}
\nc{\typeIs}{local cocycle $3$-Lie bialgebras\xspace}
\nc{\typeII}{double construction $3$-Lie bialgebra\xspace}
\nc{\typeIIs}{double construction $3$-Lie bialgebras\xspace}

\nc{\bia}{{$\mathcal{P}$-bimodule ${\bf k}$-algebra}\xspace}
\nc{\bias}{{$\mathcal{P}$-bimodule ${\bf k}$-algebras}\xspace}

\nc{\rmi}{{\mathrm{I}}}
\nc{\rmii}{{\mathrm{II}}}
\nc{\rmiii}{{\mathrm{III}}}
\nc{\pr}{{\mathrm{pr}}}
\newcommand{\huaA}{\mathcal{A}}

\nc{\pll}{\beta}
\nc{\plc}{\epsilon}

\nc{\ass}{{\mathit{Ass}}}
\nc{\lie}{{\mathit{Lie}}}
\nc{\comm}{{\mathit{Comm}}}
\nc{\dend}{{\mathit{Dend}}}
\nc{\zinb}{{\mathit{Zinb}}}
\nc{\tdend}{{\mathit{TDend}}}
\nc{\prelie}{{\mathit{preLie}}}
\nc{\postlie}{{\mathit{PostLie}}}
\nc{\quado}{{\mathit{Quad}}}
\nc{\octo}{{\mathit{Octo}}}
\nc{\ldend}{{\mathit{ldend}}}
\nc{\lquad}{{\mathit{LQuad}}}

 \nc{\adec}{\check{;}} \nc{\aop}{\alpha}
\nc{\dftimes}{\widetilde{\otimes}} \nc{\dfl}{\succ} \nc{\dfr}{\prec}
\nc{\dfc}{\circ} \nc{\dfb}{\bullet} \nc{\dft}{\star}
\nc{\dfcf}{{\mathbf k}} \nc{\apr}{\ast} \nc{\spr}{\cdot}
\nc{\twopr}{\circ} \nc{\tspr}{\star} \nc{\sempr}{\ast}
\nc{\disp}[1]{\displaystyle{#1}}
\nc{\bin}[2]{ (_{\stackrel{\scs{#1}}{\scs{#2}}})}  
\nc{\binc}[2]{ \left (\!\! \begin{array}{c} \scs{#1}\\
    \scs{#2} \end{array}\!\! \right )}  
\nc{\bincc}[2]{  \left ( {\scs{#1} \atop
    \vspace{-.5cm}\scs{#2}} \right )}  
\nc{\sarray}[2]{\begin{array}{c}#1 \vspace{.1cm}\\ \hline
    \vspace{-.35cm} \\ #2 \end{array}}
\nc{\bs}{\bar{S}} \nc{\dcup}{\stackrel{\bullet}{\cup}}
\nc{\dbigcup}{\stackrel{\bullet}{\bigcup}} \nc{\etree}{\big |}
\nc{\la}{\longrightarrow} \nc{\fe}{\'{e}} \nc{\rar}{\rightarrow}
\nc{\dar}{\downarrow} \nc{\dap}[1]{\downarrow
\rlap{$\scriptstyle{#1}$}} \nc{\uap}[1]{\uparrow
\rlap{$\scriptstyle{#1}$}} \nc{\defeq}{\stackrel{\rm def}{=}}
\nc{\dis}[1]{\displaystyle{#1}} \nc{\dotcup}{\,
\displaystyle{\bigcup^\bullet}\ } \nc{\sdotcup}{\tiny{
\displaystyle{\bigcup^\bullet}\ }} \nc{\hcm}{\ \hat{,}\ }
\nc{\hcirc}{\hat{\circ}} \nc{\hts}{\hat{\shpr}}
\nc{\lts}{\stackrel{\leftarrow}{\shpr}}
\nc{\rts}{\stackrel{\rightarrow}{\shpr}} \nc{\lleft}{[}
\nc{\lright}{]} \nc{\uni}[1]{\tilde{#1}} \nc{\wor}[1]{\check{#1}}
\nc{\free}[1]{\bar{#1}} \nc{\den}[1]{\check{#1}} \nc{\lrpa}{\wr}
\nc{\curlyl}{\left \{ \begin{array}{c} {} \\ {} \end{array}
    \right .  \!\!\!\!\!\!\!}
\nc{\curlyr}{ \!\!\!\!\!\!\!
    \left . \begin{array}{c} {} \\ {} \end{array}
    \right \} }
\nc{\leaf}{\ell}       
\nc{\longmid}{\left | \begin{array}{c} {} \\ {} \end{array}
    \right . \!\!\!\!\!\!\!}
\nc{\ot}{\otimes} \nc{\sot}{{\scriptstyle{\ot}}}
\nc{\otm}{\overline{\ot}}
\nc{\ora}[1]{\stackrel{#1}{\rar}}
\nc{\ola}[1]{\stackrel{#1}{\la}}
\nc{\pltree}{\calt^\pl}
\nc{\epltree}{\calt^{\pl,\NC}}
\nc{\rbpltree}{\calt^r}
\nc{\scs}[1]{\scriptstyle{#1}} \nc{\mrm}[1]{{\rm #1}}
\nc{\dirlim}{\displaystyle{\lim_{\longrightarrow}}\,}
\nc{\invlim}{\displaystyle{\lim_{\longleftarrow}}\,}
\nc{\mvp}{\vspace{0.5cm}} \nc{\svp}{\vspace{2cm}}
\nc{\vp}{\vspace{8cm}} \nc{\proofbegin}{\noindent{\bf Proof: }}
\nc{\proofend}{$\blacksquare$ \vspace{0.5cm}}
\nc{\freerbpl}{{F^{\mathrm RBPL}}}
\nc{\sha}{{\mbox{\cyr X}}}  
\nc{\ncsha}{{\mbox{\cyr X}^{\mathrm NC}}} \nc{\ncshao}{{\mbox{\cyr
X}^{\mathrm NC,\,0}}}
\nc{\shpr}{\diamond}    
\nc{\shprm}{\overline{\diamond}}    
\nc{\shpro}{\diamond^0}    
\nc{\shprr}{\diamond^r}     
\nc{\shpra}{\overline{\diamond}^r}
\nc{\shpru}{\check{\diamond}} \nc{\catpr}{\diamond_l}
\nc{\rcatpr}{\diamond_r} \nc{\lapr}{\diamond_a}
\nc{\sqcupm}{\ot}
\nc{\lepr}{\diamond_e} \nc{\vep}{\varepsilon} \nc{\labs}{\mid\!}
\nc{\rabs}{\!\mid} \nc{\hsha}{\widehat{\sha}}
\nc{\lsha}{\stackrel{\leftarrow}{\sha}}
\nc{\rsha}{\stackrel{\rightarrow}{\sha}} \nc{\lc}{\lfloor}
\nc{\rc}{\rfloor}
\nc{\tpr}{\sqcup}
\nc{\nctpr}{\vee}
\nc{\plpr}{\star}
\nc{\rbplpr}{\bar{\plpr}}
\nc{\sqmon}[1]{\langle #1\rangle}
\nc{\forest}{\calf}
\nc{\altx}{\Lambda_X} \nc{\vecT}{\vec{T}} \nc{\onetree}{\bullet}
\nc{\Ao}{\check{A}}
\nc{\seta}{\underline{\Ao}}
\nc{\deltaa}{\overline{\delta}}
\nc{\trho}{\tilde{\rho}}

\nc{\rpr}{\circ}
\nc{\dpr}{{\tiny\diamond}}
\nc{\rprpm}{{\rpr}}

\nc{\mmbox}[1]{\mbox{\ #1\ }} \nc{\ann}{\mrm{ann}}
\nc{\Aut}{\mrm{Aut}} \nc{\can}{\mrm{can}}
\nc{\twoalg}{{two-sided algebra}\xspace}
\nc{\colim}{\mrm{colim}}
\nc{\Cont}{\mrm{Cont}} \nc{\rchar}{\mrm{char}}
\nc{\cok}{\mrm{coker}} \nc{\dtf}{{R-{\rm tf}}} \nc{\dtor}{{R-{\rm
tor}}}
\renewcommand{\det}{\mrm{det}}
\nc{\depth}{{\mrm d}}
\nc{\Div}{{\mrm Div}} \nc{\End}{\mrm{End}} \nc{\Ext}{\mrm{Ext}}
\nc{\Fil}{\mrm{Fil}} \nc{\Frob}{\mrm{Frob}} \nc{\Gal}{\mrm{Gal}}
\nc{\GL}{\mrm{GL}} \nc{\Hom}{\mrm{Hom}} \nc{\hsr}{\mrm{H}}
\nc{\hpol}{\mrm{HP}} \nc{\id}{\mrm{id}} \nc{\im}{\mrm{im}}
\nc{\incl}{\mrm{incl}} \nc{\length}{\mrm{length}}
\nc{\LR}{\mrm{LR}} \nc{\mchar}{\rm char} \nc{\NC}{\mrm{NC}}
\nc{\mpart}{\mrm{part}} \nc{\pl}{\mrm{PL}}
\nc{\ql}{{\QQ_\ell}} \nc{\qp}{{\QQ_p}}
\nc{\rank}{\mrm{rank}} \nc{\rba}{\rm{RBA }} \nc{\rbas}{\rm{RBAs }}
\nc{\rbpl}{\mrm{RBPL}}
\nc{\rbw}{\rm{RBW }} \nc{\rbws}{\rm{RBWs }} \nc{\rcot}{\mrm{cot}}
\nc{\rest}{\rm{controlled}\xspace}
\nc{\rdef}{\mrm{def}} \nc{\rdiv}{{\rm div}} \nc{\rtf}{{\rm tf}}
\nc{\rtor}{{\rm tor}} \nc{\res}{\mrm{res}} \nc{\SL}{\mrm{SL}}
\nc{\Spec}{\mrm{Spec}} \nc{\tor}{\mrm{tor}} \nc{\Tr}{\mrm{Tr}}
\nc{\mtr}{\mrm{sk}}

\nc{\ab}{\mathbf{Ab}} \nc{\Alg}{\mathbf{Alg}}
\nc{\Algo}{\mathbf{Alg}^0} \nc{\Bax}{\mathbf{Bax}}
\nc{\Baxo}{\mathbf{Bax}^0} \nc{\RB}{\mathbf{RB}}
\nc{\RBo}{\mathbf{RB}^0} \nc{\BRB}{\mathbf{RB}}
\nc{\Dend}{\mathbf{DD}} \nc{\bfk}{{\bf k}} \nc{\bfone}{{\bf 1}}
\nc{\base}[1]{{a_{#1}}} \nc{\detail}{\marginpar{\bf More detail}
    \noindent{\bf Need more detail!}
    \svp}
\nc{\Diff}{\mathbf{Diff}} \nc{\gap}{\marginpar{\bf
Incomplete}\noindent{\bf Incomplete!!}
    \svp}
\nc{\FMod}{\mathbf{FMod}} \nc{\mset}{\mathbf{MSet}}
\nc{\rb}{\mathrm{RB}} \nc{\Int}{\mathbf{Int}}
\nc{\Mon}{\mathbf{Mon}}
\nc{\remarks}{\noindent{\bf Remarks: }}
\nc{\OS}{\mathbf{OS}} 
\nc{\Rep}{\mathbf{Rep}}
\nc{\Rings}{\mathbf{Rings}} \nc{\Sets}{\mathbf{Sets}}
\nc{\DT}{\mathbf{DT}}

\nc{\BA}{{\mathbb A}} \nc{\CC}{{\mathbb C}} \nc{\DD}{{\mathbb D}}
\nc{\EE}{{\mathbb E}} \nc{\FF}{{\mathbb F}} \nc{\GG}{{\mathbb G}}
\nc{\HH}{{\mathbb H}} \nc{\LL}{{\mathbb L}} \nc{\NN}{{\mathbb N}}
\nc{\QQ}{{\mathbb Q}} \nc{\RR}{{\mathbb R}} \nc{\BS}{{\mathbb{S}}} \nc{\TT}{{\mathbb T}}
\nc{\VV}{{\mathbb V}} \nc{\ZZ}{{\mathbb Z}}


\nc{\calao}{{\mathcal A}} \nc{\cala}{{\mathcal A}}
\nc{\calc}{{\mathcal C}} \nc{\cald}{{\mathcal D}}
\nc{\cale}{{\mathcal E}} \nc{\calf}{{\mathcal F}}
\nc{\calfr}{{{\mathcal F}^{\,r}}} \nc{\calfo}{{\mathcal F}^0}
\nc{\calfro}{{\mathcal F}^{\,r,0}} \nc{\oF}{\overline{F}}
\nc{\calg}{{\mathcal G}} \nc{\calh}{{\mathcal H}}
\nc{\cali}{{\mathcal I}} \nc{\calj}{{\mathcal J}}
\nc{\call}{{\mathcal L}} \nc{\calm}{{\mathcal M}}
\nc{\caln}{{\mathcal N}} \nc{\calo}{{\mathcal O}}
\nc{\calp}{{\mathcal P}} \nc{\calq}{{\mathcal Q}} \nc{\calr}{{\mathcal R}}
\nc{\calt}{{\mathcal T}} \nc{\caltr}{{\mathcal T}^{\,r}}
\nc{\calu}{{\mathcal U}} \nc{\calv}{{\mathcal V}}
\nc{\calw}{{\mathcal W}} \nc{\calx}{{\mathcal X}}
\nc{\CA}{\mathcal{A}}

\nc{\fraka}{{\mathfrak a}} \nc{\frakB}{{\mathfrak B}}
\nc{\frakb}{{\mathfrak b}} \nc{\frakd}{{\mathfrak d}}
\nc{\oD}{\overline{D}}
\nc{\frakF}{{\mathfrak F}} \nc{\frakg}{{\mathfrak g}}
\nc{\frakm}{{\mathfrak m}} \nc{\frakM}{{\mathfrak M}}
\nc{\frakMo}{{\mathfrak M}^0} \nc{\frakp}{{\mathfrak p}}
\nc{\frakS}{{\mathfrak S}} \nc{\frakSo}{{\mathfrak S}^0}
\nc{\fraks}{{\mathfrak s}} \nc{\os}{\overline{\fraks}}
\nc{\frakT}{{\mathfrak T}}
\nc{\oT}{\overline{T}}
\nc{\frakX}{{\mathfrak X}} \nc{\frakXo}{{\mathfrak X}^0}
\nc{\frakx}{{\mathbf x}}
\nc{\frakTx}{\frakT}      
\nc{\frakTa}{\frakT^a}        
\nc{\frakTxo}{\frakTx^0}   
\nc{\caltao}{\calt^{a,0}}   
\nc{\ox}{\overline{\frakx}} \nc{\fraky}{{\mathfrak y}}
\nc{\frakz}{{\mathfrak z}} \nc{\oX}{\overline{X}}

\font\cyr=wncyr10

\nc{\redtext}[1]{\textcolor{red}{#1}}


\title{Conformal classical Yang-Baxter equation,
$S$-equation and $\mathcal{O}$-operators}

\author{Yanyong Hong}
\address{College of Science, Zhejiang Agriculture and Forestry University,
Hangzhou, 311300, P.R.China}
\email{hongyanyong2008@yahoo.com}

\author{Chengming Bai}
\address{Chern Institute of Mathematics \& LPMC, Nankai University, Tianjin 300071, PR China}
\email{baicm@nankai.edu.cn}

\subjclass[2010]{17A30,  17B62, 17B65, 17B69}
\keywords{Lie conformal algebra, left-symmetric conformal algebra, conformal CYBE, conformal $S$-equation, $\mathcal{O}$-operator, Rota-Baxter operator}

\begin{abstract}
Conformal classical Yang-Baxter equation and $S$-equation naturally appear in the study of
Lie conformal bialgebras and left-symmetric conformal bialgebras. In this paper, they
are interpreted in terms of a kind of operators, namely, $\mathcal O$-operators in the conformal sense.
Explicitly,  the skew-symmetric part of
 a conformal linear map $T$ where $T_0=T_\lambda\mid_{\lambda=0}$ is
an  $\mathcal O$-operator in the conformal sense
is a skew-symmetric solution of conformal classical Yang-Baxter equation, whereas
the symmetric part is a symmetric solution of conformal $S$-equation.  One byproduct is that a finite left-symmetric conformal algebra which is a
free $\mathbb{C}[\partial]$-module gives a natural
$\mathcal O$-operator and hence there is a construction of solutions of conformal classical Yang-Baxter equation and
conformal $S$-equation from the former.
Another byproduct is that the non-degenerate solutions of these two equations  correspond to 2-cocycles of Lie conformal algebras and left-symmetric conformal
algebras respectively.  We also give a further study on a special class of $\mathcal{O}$-operators called Rota-Baxter operators on Lie conformal algebras and some explicit
examples are presented.

\end{abstract}

\maketitle

\section{Introduction}

The notion of Lie conformal algebra, formulated by Kac in \cite{K1,K2}, gives an
axiomatic description of the operator product expansion (or rather
its Fourier transform) of chiral fields in conformal field theory.
It appears as an useful tool to study vertex algebras (\cite{K1}) and has many applications in the theory of infinite-dimensional Lie algebras satisfying the locality property in \cite{K} and Hamiltonian formalism in the theory of nonlinear evolution equations (see \cite{Do} and the references therein, and also \cite{BDK, GD, Z, X4}). The
structure theory (\cite{DK1}), representation theory (\cite{CK1, CK2})
and cohomology theory (\cite{BKV}) of finite Lie conformal algebras
have been well developed. On the other hand, a vertex algebra is a ``combination" of a Lie conformal algebra
and another algebraic structure, namely a left-symmetric algebra, satisfying certain compatible conditions (\cite{BK}).  Moreover, for studying whether there exist compatible left-symmetric algebra structures on formal distribution Lie algebras, the definition of left-symmetric conformal algebra was introduced in \cite{HL}, which can be used to construct vertex algebras.

Motivated by the study of Lie bialgebras (\cite{Dr}), a theory of Lie conformal bialgebra was established in \cite{L}. The notion of a finite Lie conformal bialgebra which is free as a $\mathbb{C}[\partial]$-module was introduced to be equivalent to a conformal Manin triple associated to a non-degenerate symmetric invariant conformal bilinear form. The notion of conformal classical Yang-Baxter equation was also introduced to construct (coboundary) Lie conformal bialgebras and hence
as a byproduct, the conformal Drinfeld's double was constructed. Explicitly, let $R$ be a Lie conformal algebra and $r=\sum_i a_i\otimes b_i\in R\otimes R$.
Set $\partial^{\otimes^3}=\partial\otimes 1\otimes 1
+1\otimes \partial\otimes 1+1\otimes 1\otimes \partial$. The equation
\begin{eqnarray}
[[r,r]]:&=&\sum\limits_{i,j}([{a_i}_\mu a_j]\otimes b_i\otimes b_j|_{\mu=1\otimes \partial \otimes 1}-a_i\otimes [{a_j}_\mu b_i]\otimes b_j|_{\mu=1\otimes 1\otimes \partial}
-a_i\otimes a_j\otimes [{b_j}_\mu b_i]|_{\mu=1\otimes \partial \otimes 1})\nonumber\\
&=&0\;\;\text{mod}\; (\partial^{\otimes^3}) ~~~~~~~\text{in}\; R\otimes R\otimes R,
\end{eqnarray}
is called {\bf conformal classical Yang-Baxter equation (conformal CYBE)} in $R$.
Then the (skew-symmetric) solutions of conformal CYBE can be used to construct Lie conformal bialgebras and many interesting examples beyond the classical Lie bialgebras
due to this construction were given in \cite{L}. Therefore how to find solutions of
conformal CYBE becomes an important problem.

As both an analogue of the above Lie conformal bialgebra in the context of left-symmetric conformal algebra and
a conformal analogue of left-symmetric bialgebra given in \cite{Bai1}, the notion of
a finite left-symmetric conformal bialgebra which is
free as a $\mathbb{C}[\partial]$-module  was introduced in \cite{HL1}. It is equivalent to a parak$\ddot{\text{a}}$hler Lie conformal algebra
which is an analogue of the above Manin triple for Lie conformal algebras, but associated to a non-degenerate skew-symmetric conformal bilinear form satisfying the so-called ``2-cocycle" condition.
There is also the corresponding analogue of the classical Yang-Baxter equation, namely conformal $S$-equation. Explicitly, let $A$ be a left-symmetric conformal algebra and $r=\sum_{i}r_i\otimes l_i\in A\otimes A$.
Then
\begin{gather}
\{\{r,r\}\}:=\sum_{i,j}({l_j}_\mu r_i\otimes r_j\otimes l_i)|_{\mu=1\otimes \partial\otimes 1}-\sum_{i,j}(r_j\otimes {l_j}_\mu r_i\otimes l_i)|_{\mu=\partial\otimes 1\otimes1}\nonumber\\
-\sum_{i,j}(r_i\otimes r_j\otimes [{l_i}_\mu l_j])|_{\mu=\partial\otimes 1\otimes 1}=0 ~~~\text{mod}(\partial^{\otimes^3})\;\;{\rm in}\;\;A\otimes A\otimes A,
\end{gather}
is called {\bf conformal $S$-equation} in $A$.  The (symmetric) solutions of conformal $S$-equation can be used to construct left-symmetric conformal bialgebras.
Thus it is also natural and important to find solutions of conformal $S$-equation.

On the other hand, in the classical case, for the CYBE and $S$-equation in a Lie algebra and a left-symmetric algebra respectively, an important idea to find solutions is to replace the tensor form by an operator form (\cite{S, Ku, Bai, Bai1}). It is Semonov-Tian-Shansky who gave the first operator form of CYBE in a Lie algebra $\mathfrak g$ as a linear transformation on $\mathfrak g$ satisfying the condition that is known as a Rota-Baxter operator in the context of Lie algebra now (\cite{S}). The notion of $\mathcal O$-operator was introduced by Kupershmidt in \cite{Ku} replacing the linear transformation by a linear map, which is a natural generalization of CYBE. Both of them are equivalent to the tensor form of CYBE under certain conditions. Moreover, a systematic study involving  $\mathcal O$-operators and CYBE was given in \cite{Bai}, where an equivalence between the operator forms and tensor form of CYBE was obtained in a more general extent. There are two direct consequences.
\begin{enumerate}
\item There is a relationship between the CYBE and $S$-equation in terms of $\mathcal O$-operator, that is,  the skew-symmetric part of an
$\mathcal{O}$-operator gives a skew-symmetric solution of CYBE (\cite{Bai}), whereas a symmetric part of an
$\mathcal{O}$-operator gives a symmetric solution of $S$-equation (\cite{Bai1}). Thus the solutions of both CYBE and $S$-equation can be obtained from the construction of $\mathcal O$-operators.
\item A left-symmetric algebra gives a natural $\mathcal O$-operator and hence there are solutions of CYBE and $S$-equation obtained from left-symmetric algebras.
\end{enumerate}
In a summary, the operator forms ($\mathcal O$-operators) and the algebraic structures (left-symmetric algebras) behind indeed provide a practical construction for solutions of CYBE and $S$-equation.

Therefore it is natural to consider the ``conformal analogues" of the above construction, which is the main aim of this paper, that is, we study  what are the operator forms of conformal CYBE and $S$-equation. We would like to point out that the conformal generalization is not trivial. For example,
we define the conformal analogue of an $\mathcal O$-operator to be a conformal linear map $T$ where $T_0=T_\lambda\mid_{\lambda=0}$ is
an $\mathcal{O}$-operator in the conformal sense and in particular, assume $T_0\ne 0$, whereas the case $T_0=0$ is meaningless for the construction
of Lie conformal bialgebras and left-symmetric conformal bialgebras. Furthermore, many results in the conformal sense are obtained and hence these results will be useful to provide solutions of conformal CYBE and $S$-equation, to construct conformal Manin triples and parak$\ddot{\text{a}}$hler Lie conformal algebras and then to study some
related geometry in the conformal sense.

This paper is organized as follows. In Section 2, we recall some necessary definitions, notations and some results about Lie conformal algebras and left-symmetric conformal algebras. In Section 3, the operator forms of conformal CYBE are investigated. We introduce the definitions of Rota-Baxter operator and $\mathcal{O}$-operator of a Lie conformal algebra and obtain some relations between Rota-Baxter operator, $\mathcal{O}$-operator and
conformal CYBE. Moreover,  a relation between the non-degenerate skew-symmetric solutions of conformal CYBE and 2-cocycles of Lie conformal algebras is presented in terms of $\mathcal{O}$-operators. In Section 4, we study the relations between conformal CYBE and left-symmetric conformal algebras. Section 5 is devoted to investigating the operator forms of conformal $S$-equation. Moreover,
similar results as those in the case of conformal CYBE are obtained. In Section 6, Rota-Baxter operators on a class of Lie conformal algebras named quadratic Lie conformal algebras are studied.

Throughout this paper, denote by $\mathbb{C}$ the field of complex
numbers; $\mathbb{N}$ the set of natural numbers, i.e.
$\mathbb{N}=\{0, 1, 2,\cdots\}$; $\mathbb{Z}$ the set of integer
numbers. All tensors over $\mathbb{C}$ are denoted by $\otimes$.
Moreover, if $A$ is a vector space, the space of polynomials of $\lambda$ with coefficients in $A$ is denoted by $A[\lambda]$.

\section{Preliminaries on  conformal algebras}
In this section, we recall some definitions, notations and results about conformal algebras. Most of the results in this section can be found in \cite{K1,HL}.

\begin{defi}\label{def1}{\rm

A {\bf conformal algebra} $R$ is a $\mathbb{C}[\partial]$-module
endowed with a $\mathbb{C}$-bilinear map $R\times R\rightarrow
 R[\lambda]$ denoted by $a\times b\rightarrow
a_{\lambda} b$ satisfying
\begin{eqnarray}
\partial a_{\lambda}b=-\lambda a_{\lambda}b,  \quad
a_{\lambda}\partial b=(\partial+\lambda)a_{\lambda}b.\end{eqnarray}

A {\bf Lie conformal algebra} $R$ is a conformal algebra with the $\mathbb{C}$-bilinear
map $[\cdot_\lambda \cdot]: R\times R\rightarrow  R[\lambda]$ satisfying
\begin{eqnarray*}
&&[a_\lambda b]=-[b_{-\lambda-\partial}a],~~~~\text{(skew-symmetry)}\\
&&[a_\lambda[b_\mu c]]=[[a_\lambda b]_{\lambda+\mu} c]+[b_\mu[a_\lambda c]],~~~~~~\text{(Jacobi identity)}
\end{eqnarray*}
for $a$, $b$, $c\in R$.

A {\bf left-symmetric conformal algebra} $R$ is a conformal algebra with the $\mathbb{C}$-bilinear
map $\cdot_\lambda \cdot: R\times R\rightarrow R[\lambda]$ satisfying
\begin{eqnarray}
(a_{\lambda}b)_{\lambda+\mu}c-a_{\lambda}(b_\mu
c)=(b_{\mu}a)_{\lambda+\mu}c-b_\mu(a_\lambda c),
\end{eqnarray}
for $a$, $b$, $c\in R$.}
\end{defi}

A  conformal algebra
is called {\bf finite} if it is finitely generated as a
$\mathbb{C}[\partial]$-module. The {\bf rank} of a conformal
algebra $R$ is its rank as a $\mathbb{C}[\partial]$-module.


\begin{pro}
If $A$ is a left-symmetric conformal algebra, then the
$\lambda$-bracket
\begin{equation}\label{3}[a_\lambda b]=a_\lambda
b-b_{-\lambda-\partial}a,~~~\text{for any }~~~a, ~~b \in A,\end{equation}
defines a Lie conformal algebra $\mathfrak{g}(A)$, which is called {\bf the
sub-adjacent Lie conformal algebra of $A$}.
In this case,  $A$ is also called {\bf a
compatible left-symmetric conformal algebra structure on the Lie
conformal algebra $\mathfrak{g}(A)$}.
\end{pro}

\begin{defi}\label{def:module} {\rm
A {\bf module $M$ over a Lie conformal algebra $R$} is a $\mathbb{C}[\partial]$-module endowed with a $\mathbb{C}$-bilinear map
$R\times M\longrightarrow M[\lambda]$, $(a, v)\mapsto a_\lambda v$, satisfying the following axioms $(a, b\in R, v\in M)$:\\
(LM1)$\qquad\qquad (\partial a)_\lambda v=-\lambda a_\lambda v,~~~a_\lambda(\partial v)=(\partial+\lambda)a_\lambda v,$\\
(LM2)$\qquad\qquad [a_\lambda b]_{\lambda+\mu}v=a_\lambda(b_\mu v)-b_\mu(a_\lambda v).$}
\end{defi}

An $R$-module $M$ is called {\bf finite} if it is finitely generated as a $\mathbb{C}[\partial]$-module.

\begin{defi}{\rm
Let $U$ and $V$ be two $\mathbb{C}[\partial]$-modules. A {\bf conformal linear map} from $U$ to $V$ is a $\mathbb{C}$-linear map $a: U\rightarrow V[\lambda]$, denoted by $a_\lambda: U\rightarrow V$, such that $[\partial, a_\lambda]=-\lambda a_\lambda$. Denote the $\mathbb C$-vector space of all such maps by $\text{Chom}(U,V)$. It has a canonical structure of a $\mathbb{C}[\partial]$-module:
$$(\partial a)_\lambda =-\lambda a_\lambda.$$ Define the {\bf conformal dual} of a $\mathbb{C}[\partial]$-module $U$ as $U^{\ast c}=\text{Chom}(U,\mathbb{C})$, where $\mathbb{C}$ is viewed as the trivial $\mathbb{C}[\partial]$-module, that is
$$U^{\ast c}=\{a:U\rightarrow \mathbb{C}[\lambda]~~|~~\text{$a$ is $\mathbb{C}$-linear and}~~a_\lambda(\partial b)=\lambda a_\lambda b\}.$$}
\end{defi}

Let $U$ and $V$ be finite modules over a Lie conformal algebra $R$. Then the $\mathbb{C}[\partial]$-module $\text{Chom}(U,V)$ has an $R$-module structure defined by:
\begin{eqnarray}\label{oo2}
(a_\lambda \varphi)_\mu u=a_\lambda(\varphi_{\mu-\lambda}u)-\varphi_{\mu-\lambda}(a_\lambda u).\end{eqnarray}
for $a\in R, \varphi\in \text{Chom}(U,V), u\in U$.
Hence, one special case is the contragradient conformal $R$-module $U^{\ast c}$, where $\mathbb{C}$ is viewed as the trivial $R$-module and $\mathbb{C}[\partial]$-module. In particular, when $U=V$, set $\text{Cend}(V)=\text{Chom}(V,V)$. The $\mathbb{C}[\partial]$-module $\text{Cend}(V)$ has a canonical structure of an associative conformal algebra defined by
\begin{eqnarray}
(a_\lambda b)_\mu v=a_\lambda (b_{\mu-\lambda} v),~~~~\text{ $a$, $b\in \text{Cend}(V)$, $v\in V$.}\end{eqnarray}
Therefore $\text{gc}(V):=\text{Chom}(V,V)$ has a Lie conformal algebra structure defined by
\begin{eqnarray}
[a_\lambda b]_\mu v=a_\lambda (b_{\mu-\lambda} v)-b_{\mu-\lambda}(a_{\lambda} v), ~~~ a,~b\in \text{gc}(V), v\in V.\end{eqnarray}
$\text{gc}(V)$ is called the {\bf general Lie conformal algebra} of $V$.

\begin{rmk}\label{rrm}{\rm
By Definition~\ref{def:module}, it is easy to see that a module over a Lie conformal algebra $R$ in a finite  $\mathbb{C}[\partial]$-module $V$ is the same as a homomorphism of Lie conformal algebra $\rho: R\rightarrow \text{gc}(V)$ which is called a {\bf representation} of $R$. Denote the {\bf adjoint representation} of $R$ by $ad$, i.e. $ad(a)_\lambda b=[a_\lambda b]$, where $a$, $b\in R$. Moreover, the contragradient conformal $R$-module
$V^{\ast c}$ is the same as the representation $\rho^\ast: R\rightarrow
\text{gc}(V^{\ast c})$ which is dual to $\rho$. By Eq.~(\ref{oo2}), the relation
between $\rho^\ast$ and $\rho$ is given as follows.
\begin{eqnarray*}
(\rho^\ast(a)_\lambda \varphi)_\mu u=-\varphi_{\mu-\lambda}(\rho(a)_\lambda u),\;\;\forall a\in R, \varphi\in V^{\ast c}, u\in V.
\end{eqnarray*}}
 \end{rmk}

\begin{pro}{\rm (\cite{HL1})}
Let $R$ be a Lie conformal algebra. Let $V$ be a $\mathbb{C}[\partial]$-module of finite rank and
$\rho: R\rightarrow gc(V)$ be a representation of $R$. Then $R\oplus V$ is endowed with a $\mathbb{C}[\partial]$-module
structure given by
$$\partial (a+v)=\partial a+\partial v,\;\;\forall a\in R, v\in V.$$
Hence the $\mathbb{C}[\partial]$-module $R\oplus V$ is endowed with a  Lie conformal algebra structure as follows.
\begin{eqnarray*}
[{(a+u)}_\lambda (b+v)]=[a_\lambda b]+\rho(a)_\lambda v-\rho(b)_{-\lambda-\partial} u,~~~~~~~\text{for any $a$, $b\in R$ and $u$, $v\in V$.}
\end{eqnarray*}
 This Lie conformal algebra is called {\bf the semi-direct sum of $R$ and $V$},
denoted by $R\ltimes_\rho V$.
\end{pro}


The tensor product $U\otimes V$ can be naturally endowed with an $R$-module structure as follows.
$$\partial(u\otimes v)=\partial u\otimes v+u\otimes \partial v,$$
and
$$r_\lambda(u\otimes v)=r_\lambda u\otimes v+u\otimes r_\lambda v,$$
where $u\in U$, $v\in V$ and $r\in R$.
\begin{pro}\label{prop1} {\rm (\cite{BKL})}
Let $U$ and $V$ be two $R$-modules. Suppose that $U$ is a $\mathbb{C}[\partial]$-module of finite rank. Then
$U^{\ast c}\otimes V\cong \text{Chom}(U,V)$ as $R$-modules with the identification $(f\otimes v)_\lambda (u)=f_{\lambda+\partial}(u)v$ where
$f\in U^{\ast c}$, $u\in U$ and $v\in V$.
\end{pro}


\begin{defi}\label{555}{\rm
A  {\bf module $M$ over a left-symmetric conformal algebra $A$} is a $\mathbb{C}[\partial]$-module with two $\mathbb{C}$-bilinear maps $A\times M\rightarrow M[\lambda]$, $a\times v\rightarrow a_\lambda v$ and $M\times A\rightarrow M[\lambda]$, $v\times a\rightarrow v_\lambda a$ such that
$$(\partial a)_\lambda v=[\partial, a_\lambda]v=-\lambda a_\lambda v,\quad (\partial v)_\lambda a=[\partial, v_\lambda]a=-\lambda v_\lambda a,$$
$$(a_\lambda b)_{\lambda+\mu}v-a_\lambda(b_\mu v)=(b_\mu a)_{\lambda+\mu}v-b_\mu (a_\lambda v),$$
$$(a_\lambda v)_{\lambda+\mu}b-a_\lambda(v_\mu b)=(v_\mu a)_{\lambda+\mu}b-v_\mu(a_\lambda b)$$
hold for $a$, $b\in A$ and $v\in M$.}
\end{defi}
Similarly, an $A$-module is called {\bf finite} if it is finitely generated as a
$\mathbb{C}[\partial]$-module.

\begin{rmk}{\rm
Suppose $M$ is finite. Let $a_\lambda v=l_A(a)_\lambda v$ and $v_\lambda a=r_A(a)_{-\lambda-\partial} v$. It is easy to show that the  structure of a module $M$ over a left-symmetric conformal algebra $A$ is the same as a pair $\{l_A, r_A\}$, where $l_A, r_A: A\rightarrow \text{Cend}(M)$ are two $\mathbb{C}[\partial]$-module homomorphisms such that the following conditions hold
\begin{eqnarray}\label{101}
l_A(a_\lambda b)_{\lambda+\mu}v-l_A(a)_\lambda(l_A(b)_\mu v)=l_A(b_\mu a)_{\lambda+\mu}v-l_A(b)_\mu(l_A(a)_\lambda v),\end{eqnarray}
\begin{eqnarray}\label{102}
r_A(b)_{-\lambda-\mu-\partial}(l_A(a)_\lambda v)-l_A(a)_\lambda(r_A(b)_{-\mu-\partial}v)
=r_A(b)_{-\lambda-\mu-\partial}(r_A(a)_\lambda v)-r_A(a_\lambda b)_{-\mu-\partial} v,\end{eqnarray}
for any $a$, $b\in A$ and $v\in M$. Denote this module by $(M,l_A, r_A)$.}
\end{rmk}

Throughout this paper, we mainly deal with $\mathbb{C}[\partial]$-modules which are finitely generated. So, for convenience,  we use the notions of representations of conformal algebras instead of those of modules of conformal algebras.

\begin{pro}{\rm (\cite{HL1})}
Let $(M,l_A, r_A)$ be a finite module over a left-symmetric conformal algebra $A$. Then\\
(i) $l_A:A\rightarrow \text{gc}(M)$ is a representation of the sub-adjacent Lie conformal algebra $\mathfrak{g}(A)$.\\
(ii) $\rho=l_A-r_A$ is a representation of the Lie conformal algebra $\mathfrak{g}(A)$.\\
(iii) For any representation $\sigma:\mathfrak{g}(A)\rightarrow \text{gc}(M)$ of the Lie conformal algebra $\mathfrak{g}(A)$,
$(M,\sigma,0)$ is an $A$-module.
\end{pro}

\begin{cor}
Let $A$ be a finite left-symmetric conformal algebra. Define two $\mathbb{C}[\partial]$-module homomorphisms $L_A$ and $R_A$ from $A$ to $\text{Cend}(A)$ by $L_A(a)_\lambda b=a_\lambda b$ and $R_A(a)_\lambda b=b_{-\lambda-\partial}a$ for any $a$, $b\in A$. Then $L_A:A\rightarrow \text{gc}(A)$ and $\rho=L_A-R_A$ are two representations of the Lie conformal algebra $\mathfrak{g}(A)$.
\end{cor}
\begin{rmk}
{\rm $L_A$ is called the {\bf regular representation} of $\mathfrak{g}(A)$.}
\end{rmk}

\begin{pro}{\rm (\cite{HL1})}
Let $A$ be a left-symmetric conformal algebra and  $(M,l_A, r_A)$ be a module of $A$. Then the $\mathbb{C}[\partial]$-module $A\oplus M$ is a left-symmetric conformal algebra with the following $\lambda$-product
\begin{eqnarray}
(a+u)_\lambda (b+v)=a_\lambda b+l_A(a)_\lambda v+r_A(b)_{-\lambda-\partial}u,~~a,~b\in A,~~u,~v\in M.\end{eqnarray}
Denote it by $A\ltimes_{l_A,r_A} M$, which is called {\bf the semi-direct sum of $A$ and $M$}.
\end{pro}


Finally, let us recall the definition of coefficient algebra of a conformal algebra. Let $\chi$ be a variety of algebras (Lie,
left-symmetric, etc)
and $R$ be a $\chi$-conformal algebra.
There is an associated $\chi$-algebra
 constructed as follows.
Set $$a_\lambda b=\sum_{n\in \mathbb{N}}\frac{\lambda^n}{n!}a_{(n)}b,$$
where $a_{(n)}b$ is called {\bf the $n$-th product} of $a$ and $b$.
Let Coeff$R$ be the quotient
of the vector space with basis $a_n$ $(a\in R, n\in\mathbb{Z})$ by
the subspace spanned over $\mathbb{C}$ by
elements:
$$(\alpha a)_n-\alpha a_n,~~(a+b)_n-a_n-b_n,~~(\partial
a)_n+na_{n-1},~~~\text{where}~~a,~~b\in R,~~\alpha\in \mathbb{C},~~n\in
\mathbb{Z}.$$ The operation on $\text{Coeff}R$ is defined as follows.
\begin{eqnarray}\label{eq1}
a_m\cdot b_n=\sum_{j\in \mathbb{N}}\left(\begin{array}{ccc}
m\\j\end{array}\right)(a_{(j)}b)_{m+n-j}.\end{eqnarray} Then
$\text{Coeff}R$ is a $\chi$-algebra (\cite{K1}), which
 is called the {\bf coefficient algebra} of $R$.

\section{The operator forms of conformal CYBE}
In this section, we study the operator forms of conformal CYBE and give a relation between the non-degenerate skew-symmetric solutions of conformal CYBE and 2-cocycles of Lie conformal algebras.

Let $R$ be a Lie conformal algebra and $r=\sum_{i} a_i\otimes b_i\in R\otimes R$. Set $r^{21}=\sum_{i} b_i\otimes a_i$.  We say $r$ is {\bf skew-symmetric} if $r=-r^{21}$, whereas $r$ is called {\bf symmetric} if $r=r^{21}$.


Let $V$ be free and of finite rank as a $\mathbb{C}[\partial]$-module. Set $\{v_i\}\mid_{i=1,\cdots,m}$ be a $\mathbb{C}[\partial]$-basis of $V$. Obviously,
as an $R$-module, $V\cong {V^{\ast c}}^{\ast c}$ though the $\mathbb{C}[\partial]$-module homomorphism $v_i\rightarrow {v_i}^{\ast\ast}$ where ${y^\ast}_\lambda(v_i)={v_i^{\ast\ast}}_{-\lambda-\partial}(y^\ast)$ for any $y^\ast \in V^{\ast c}$.

Now suppose that $R$ is a finite Lie conformal algebra which is free as a $\mathbb{C}[\partial]$-module. Then  by the discussion above,
as $R$-modules, $$R\otimes R\cong {R^{\ast c}}^{\ast c}\otimes R \cong
\text{Chom}(R^{\ast c},R).$$ Define
$$\{u,a\}_\lambda=u_\lambda(a),\;\; {\rm and}\;\;
\{u\otimes v, a\otimes b\}_{(\lambda,\mu)}=\{u,a\}_\lambda \{v,b\}_\mu.$$ where $a$, $b\in R$ and $u$, $v\in R^{\ast c}$.
Then  by Proposition~ \ref{prop1}, for any $r\in R\otimes R$, we associate
a conformal linear map $T^r\in \text{Chom}(R^{\ast c},R)$ as follows.
\begin{eqnarray}\label{eqn5}
\{ f, T^r_{-\mu-\partial}(g)\}_\lambda =\{ g\otimes f, r\}_{(\mu,\lambda)},~~~~~~~~~~~~\text{$f$, $g\in R^{\ast c}$.}
\end{eqnarray}
Set $r=\sum_{i} a_i\otimes b_i\in R\otimes R$. Then  by Eq.~(\ref{eqn5}), we get
$$T_{\lambda}^r(u)=\sum_i \{ u, a_i\}_{-\lambda-\partial} b_i, ~~~~u\in R^{\ast c}.$$

\begin{thm}\label{tthe3}
Let $R$ be a finite Lie conformal algebra which is free as a $\mathbb{C}[\partial]$-module and $r\in R\otimes R$ be skew-symmetric. Then $r$ is a solution of
conformal CYBE if and only if the $T^r\in \text{Chom}(R^{\ast c},R)$ corresponding to $r$ satisfies
\begin{eqnarray}
[T_0^r(u)_\lambda T_0^r(v)]= T_0^r(ad^{\ast}(T_0^r(u))_\lambda v
-ad^{\ast}(T_0^r(v))_{-\lambda-\partial} u),\;\;\forall u, v\in R^{\ast c},
\end{eqnarray}
where $T_0^r=T_\lambda^r\mid_{\lambda=0}$.

\end{thm}
\begin{proof}
Define $\langle a, u\rangle_\lambda=\{u,a\}_{-\lambda}=u_{-\lambda}(a)$ where $a\in R$ and $u\in R^{\ast c}$. Obviously, it is easy to see that
\begin{eqnarray}\label{kh1}
\langle \partial a, u\rangle_\lambda=-\lambda\langle a, u\rangle_\lambda.
\end{eqnarray}
Moreover, we define
$$\langle a\otimes b\otimes c, u\otimes v\otimes w\rangle_{(\lambda, \nu,\theta)}=\langle a,u\rangle_\lambda \langle b,v\rangle_\nu\langle c,w\rangle_\theta,$$
where $a$, $b$, $c\in R$, and $u$, $v$, $w\in R^{\ast c}$. Moreover, by the definition of dual representation in Remark \ref{rrm}, we can easily get
\begin{eqnarray}\label{x2}
\langle [a_\mu b], u\rangle_\lambda=\langle a, ad^\ast(b)_{\lambda-\partial} u\rangle_\mu,
\end{eqnarray}
for any $a$, $b\in R$, and $u\in R^{\ast c}$.

Set $r=\sum_{i} a_i\otimes b_i\in R\otimes R$.
By the discussion above, $T^r \in \text{Chom}(R^{\ast c},R)$ corresponding to $r$ is given by
$$T_{\lambda}^r(u)=\sum_i \langle a_i, u\rangle_{\lambda+\partial} b_i, ~~~~u\in R^{\ast c}.$$
Similarly, since $r$ is skew-symmetric, one can obtain
$T_{\lambda}^r(v)=-\sum_i \langle b_i, v\rangle_{\lambda+\partial} a_i, ~~~~v\in R^{\ast c}$.

Then we consider
\begin{eqnarray}
\label{q1}\langle [[r,r]]~~\text{mod}~(\partial^{\otimes^3}), u\otimes v\otimes w\rangle_{(\lambda,\nu,\theta)}=0 ~~~\text{ for any $u$, $v$, $w\in R^{\ast c}$.}
\end{eqnarray}
By Eq.~(\ref{kh1}), Eq.~(\ref{q1}) is equivalent to the following equality
\begin{eqnarray}
\label{q2}\langle [[r,r]], u\otimes v\otimes w\rangle_{(\lambda,\nu,\theta)}=0 ~~~~ \text{mod}~(\lambda+\nu+\theta) ~~~\text{ for any $u$, $v$, $w\in R^{\ast c}$.}
\end{eqnarray}
By a direct computation, we have
\begin{eqnarray*}
&&\langle \sum_{i,j}[{a_i}_\mu a_j]\otimes b_i\otimes b_j\mid_{\mu=1\otimes \partial \otimes 1}, u\otimes v\otimes w\rangle_{(\lambda,\nu,\theta)}\\
&&=\sum_{i,j}\langle [{a_i}_{-\nu} a_j], u\rangle_\lambda \langle b_i, v\rangle_\nu \langle b_j, w\rangle_\theta\\
&&=\sum_{i,j} \langle [(\langle b_i, v\rangle_\nu {a_i})_{-\nu} a_j], u\rangle_\lambda \langle b_j, w\rangle_\theta=-\sum_{i,j}\langle [T_0^r(v)_{-\nu} a_j], u\rangle_{\lambda}\langle b_j, w\rangle_\theta\\
&&=\sum_{i,j}\langle [{a_j}_{\lambda+\nu}T_0^r(v)],u\rangle_\lambda \langle b_j, w\rangle_\theta=\sum_{i,j}\langle a_j, ad^\ast (T_0^r(v))_{-\nu} u\rangle_{\lambda+\nu}
\langle b_j, w\rangle_\theta\\
&&= \sum_{i,j}\langle \langle a_j, ad^\ast (T_0^r(v))_{-\nu} u\rangle_{\lambda+\nu} b_j, w\rangle_\theta=\langle T_{\lambda+\nu-\partial}^r(ad^\ast (T_0^r(v))_{-\nu} u), w\rangle_\theta\\
&&=\langle T_{\lambda+\nu+\theta}^r(ad^\ast (T_0^r(v))_{-\nu} u), w\rangle_\theta.
\end{eqnarray*}
Similarly, one can get
\begin{eqnarray*}
\langle \sum_{i,j}
a_i\otimes[{a_j}_\mu b_i]\otimes b_j\mid_{\mu=1\otimes 1\otimes \partial}, u\otimes v\otimes w\rangle_{(\lambda,\nu,\theta)}
=\langle T_0^r(ad^\ast(T_{\lambda+\nu+\theta}^r(u))_{\nu+\theta} v), w \rangle_\theta,
\end{eqnarray*}
and
\begin{eqnarray*}
\langle \sum_{i,j}a_i\otimes a_j \otimes [{b_j}_\mu b_i]\mid_{\mu=1\otimes \partial\otimes 1}, u\otimes v\otimes w\rangle_{(\lambda,\nu,\theta)}
=-\langle [T_{\lambda+\nu+\theta}^r(u)_{\nu+\theta}T_0^r(v)], w\rangle_\theta.
\end{eqnarray*}
By Eq.~(\ref{q2}) and the above discussion, the conformal CYBE is equivalent to
\begin{gather}
\langle [T_{\lambda+\nu+\theta}^r(u)_{\nu+\theta}T_0^r(v)]-T_0^r(ad^\ast(T_{\lambda+\nu+\theta}^r(u))_{\nu+\theta} v)\nonumber\\
\label{q3}-T_{\lambda+\nu+\theta}^r(ad^\ast (T_0^r(v))_{-\nu} u), w\rangle_\theta=0 ~~~\text{mod}~~(\lambda+\nu+\theta).
\end{gather}
Therefore $(\ref{q3})$ is equivalent to the following equality
\begin{eqnarray}
\label{q4}[T_0^r(u)_{-\lambda} T_0^r(v)]=T_0^r(ad^\ast(T_{0}^r(u))_{-\lambda} v)+T_{0}^r(ad^\ast (T_0^r(v))_{\lambda-\partial} u).
\end{eqnarray}
Then  we get the conclusion replacing $-\lambda$ by $\lambda$.
\end{proof}

\begin{defi} \label{def:bl} {\rm (\cite{L})\quad
A {\bf conformal bilinear form} on $R$ is a $\mathbb{C}$-bilinear map
$\langle,\rangle_\lambda: R\times R\rightarrow \mathbb{C}[\lambda]$ satisfying
\begin{eqnarray}
\langle \partial a, b\rangle_\lambda =-\lambda\langle a,b \rangle_\lambda
=-\langle a,\partial b\rangle_\lambda~~~~~\text{ for all $a$, $b\in R$.}
\end{eqnarray}
If $\langle a, b\rangle_\lambda=\langle b,a\rangle_{-\lambda}$ for all
$a$, $b\in R$, we say this conformal bilinear form is {\bf symmetric}.
$\langle,\rangle_\lambda$ is called {\bf invariant} if for any $a$, $b$, $c\in R$,
\begin{eqnarray}
\langle [a_\mu b], c\rangle_\lambda=\langle a, [b_{\lambda-\partial}c]\rangle_\mu=-\langle a, [c_{-\lambda}b]\rangle_\mu.
\end{eqnarray}
Suppose $R$ is a free and of finite rank $\mathbb{C}[\partial]$-module. Given a conformal bilinear form on $R$. If the $\mathbb{C}[\partial]$-module homomorphism $L: R\rightarrow R^{\ast c},
a\rightarrow L_a$ given by $(L_a)_\lambda b=\langle a, b\rangle_\lambda$, $b\in R$, is an isomorphism, then  we call the bilinear form {\bf non-degenerate}.}
\end{defi}

Let $R$ have a non-degenerate conformal bilinear form. For any $a\otimes b$, $c\otimes d\in R\otimes R$, set
\begin{eqnarray}
\langle a\otimes b, c\otimes d\rangle_{(\lambda,\mu)}
=\langle a, c\rangle_\lambda \langle b, d\rangle_\mu.
\end{eqnarray}
Let $r=\sum_{i} a_i\otimes b_i\in R\otimes R$. Define a linear map $P^r: R\rightarrow R[\lambda]$ as
\begin{eqnarray}
\langle r, u\otimes v\rangle_{(\lambda,\mu)}=\langle P_{\lambda-\partial}^r(u), v\rangle_\mu.
\end{eqnarray}
It is easy to see that $P^r\in \text{Cend}(R)$.

\begin{cor}\label{t1}
Let $R$ be a finite Lie conformal algebra which is free as a $\mathbb{C}[\partial]$-module.
Suppose that there exists a non-degenerate symmetric invariant conformal bilinear form on $R$ and $r\in R\otimes R$ is skew-symmetric. Then $r$ is a solution
of conformal CYBE if and only if the element $P^r\in
\text{Cend}(R)$ corresponding to $r$ satisfies the following equality
\begin{eqnarray}
\label{et1}[P_0^r(a)_\lambda P_0^r(b)]=P_0^r([a_\lambda P_0^r(b)])+P_0^r([P_0^r(a)_\lambda b]),~~~~\forall a, b\in R,
\end{eqnarray}
where $P_0^r=P_\lambda^r\mid_{\lambda=0}$.
\end{cor}
\begin{proof}
Since $R$ has a non-degenerate symmetric invariant conformal bilinear form,  $R^{\ast c}$ is isomorphic to  $R$ as a $\mathbb{C}[\partial]$-module through this conformal bilinear form. Then  it can be directly obtained from Theorem \ref{tthe3}.
\end{proof}

Note that $P_0^r$ is a $\mathbb{C}[\partial]$-module homomorphism, which
motivates us to give the following definition.

\begin{defi}\label{RB0}{\rm
Let $R$ be a Lie conformal algebra. If $T:R\rightarrow R$ is a $\mathbb{C}[\partial]$-module homomorphism satisfying
\begin{eqnarray}\label{oo3}
[T(a)_\lambda T(b)]=T([a_\lambda T(b)])+T([T(a)_\lambda b]),~~\forall~~\text{$a$, $b\in R$,}
\end{eqnarray}
then $T$ is called a {\bf Rota-Baxter operator (of weight 0) } on $R$.}
\end{defi}

\begin{defi} {\rm Let $R$ be a Lie conformal algebra and $\rho: R\rightarrow gc(V)$ be a representation. If a $\mathbb{C}[\partial]$-module homomorphism $T: V\rightarrow R$ satisfies
\begin{eqnarray}\label{eqn3}
[T(u)_\lambda T(v)]=T(\rho(T(u))_\lambda v-\rho(T(v))_{-\lambda-\partial}u),~~\forall~~u,~v\in V,
\end{eqnarray}
then $T$ is called an {\bf $\mathcal{O}$-operator} associated with $\rho$.}
\end{defi}
\begin{rmk}
{\rm By Theorem \ref{tthe3}, a skew-symmetric solution of conformal CYBE in $R$ is equivalent to  $T\in \text{Chom}(R^{\ast c},R)$ where $T_0$ is an $\mathcal{O}$-operator associated to $ad^\ast$ and $R$ is a finite Lie conformal algebra and free as a $\mathbb{C}[\partial]$-module.}
\end{rmk}

Next, we study the $\mathcal{O}$-operators of Lie conformal algebras
in a more general extent.

In the following, let $R$ and $V$ be free $\mathbb{C}[\partial]$-modules of finite ranks.
Let $\{e_i\}_{i=1}^n$ be a $\mathbb{C}[\partial]$-basis of $R$,
$\{v_j\}_{j=1}^m$ be a $\mathbb{C}[\partial]$-basis of $V$ and $\{v_j^\ast\}_{j=1}^m$ be the dual $\mathbb{C}[\partial]$-basis in $V^{\ast c}$. Then  there is
a natural representation $\rho^\ast: R\rightarrow gc(V^{\ast c})$ which is dual to $\rho$ given by
\begin{eqnarray}\label{eqn2}
\rho^\ast(e_i)_\lambda v_j^{\ast}=-\sum_{k=1}^m{v_j^\ast}_{-\lambda-\partial}
(\rho(e_i)_\lambda (v_k))v_k^\ast.
\end{eqnarray}

By Proposition \ref{prop1}, as $R$-modules, $\text{Chom}(V, R)\cong V^{\ast c}\otimes R\cong R\otimes V^{\ast c}$. Therefore through this isomorphism, any conformal linear map
$T\in \text{Chom}(V, R)$ corresponds to an element $r_{T}\in
R\otimes V^{\ast c}\in R\ltimes_{\rho^\ast}V^{\ast c}\otimes R\ltimes_{\rho^\ast}V^{\ast c}$. Set $T_\lambda(v_i)=\sum_{j=1}^na_{ij}(\lambda,\partial)e_j$ for $i=1$,
$\cdots$, $m$. Then  by the definition of the isomorphism, we get
\begin{eqnarray}
\label{eqn1}r_T=\sum_{i=1}^m\sum_{j=1}^n a_{ij}(-1\otimes \partial-\partial\otimes 1,\partial\otimes 1)e_j\otimes v_i^\ast.
\end{eqnarray}

\begin{thm}\label{t2}
With the conditions above, $r=r_T-r_T^{21}$ is a solution of conformal CYBE in $R\ltimes_{\rho^\ast}V^{\ast c}$ if and only if
for $T\in \text{Chom}(V, R)$, $T_0=T_\lambda\mid_{\lambda=0}$ is an $\mathcal{O}$-operator.
\end{thm}
\begin{proof}
By Eq.~(\ref{eqn1}), we get
\begin{eqnarray*}
r=\sum_{i,j}a_{ij}(-\partial\otimes 1-1\otimes \partial,\partial\otimes 1)e_j\otimes v_i^\ast-\sum_{i,j}a_{ij}(-\partial\otimes 1-1\otimes \partial,1\otimes \partial)v_i^\ast \otimes e_j.
\end{eqnarray*}
Then by a direct computation, we get
\begin{eqnarray*}
[[r,r]]
&=&(\sum_{i,j,k,l} a_{ij}(0,-1\otimes \partial\otimes 1)a_{kl}(0,-1\otimes 1\otimes \partial)[{e_j}_\mu e_l]\otimes v_i^\ast \otimes v_k^\ast\\
&&-\sum_{i,j,k,l}a_{ij}(0,1\otimes \partial \otimes 1) a_{kl}(0,-1\otimes 1\otimes \partial)[{v_i^\ast}_\mu e_l]\otimes e_j\otimes v_k^\ast\\
&&-\sum_{i,j,k,l}a_{ij}(0,-1\otimes \partial \otimes 1)a_{kl}(0,1\otimes 1\otimes \partial)[{e_j}_\mu v_k^\ast]\otimes v_i^\ast \otimes e_l)\mid_{\mu=1\otimes \partial \otimes 1}\\
&&-(\sum_{i,j,k,l}a_{ij}(0,\partial\otimes 1\otimes 1)a_{kl}(0,-1\otimes 1\otimes \partial)e_j\otimes [{e_l}_\mu v_i^\ast] \otimes v_k^\ast\\
&&-\sum_{i,j,k,l}a_{ij}(0,-\partial\otimes 1\otimes 1)a_{kl}(0,-1\otimes 1\otimes \partial)v_i^\ast\otimes [{e_l}_\mu e_j]\otimes v_k^\ast\\
&&+\sum_{i,j,k,l}a_{ij}(0,-\partial\otimes 1\otimes 1)a_{kl}(0,1\otimes 1\otimes \partial)v_i^\ast\otimes [{v_k^\ast}_\mu e_j]\otimes v_l)\mid_{\mu=1\otimes 1\otimes \partial}\\
&&-(-\sum_{i,j,k,l}a_{ij}(0,\partial\otimes 1\otimes 1)a_{kl}(0,-1\otimes \partial\otimes 1)e_j\otimes v_k^\ast\otimes [{e_l}_\mu v_i^\ast]\\
&&-\sum_{i,j,k,l}a_{ij}(0,-\partial\otimes 1\otimes 1)a_{kl}(0,1\otimes \partial\otimes 1)v_i^\ast\otimes e_l\otimes [{v_k^\ast}_\mu e_j]\\
&&+\sum_{i,j,k,l}a_{ij}(0,-\partial\otimes 1\otimes 1)a_{kl}(0,-1\otimes \partial\otimes 1)v_i^\ast\otimes v_k^\ast\otimes [{e_l}_\mu e_j])\mid_{\mu=1\otimes \partial \otimes 1}~~~\text{mod}~~(\partial^{\otimes^3}).
\end{eqnarray*}
Note that $T_0(v_i)=\sum_{j=1}^na_{ij}(0,\partial)e_j$.
Thus we obtain
\begin{eqnarray*}
[[r,r]]
&=&\sum_{i,k}([T_0(v_i)_\mu T_0(v_k)]\otimes v_i^\ast \otimes v_k^\ast
+\rho^\ast(T_0(v_k))_{-\mu-\partial}v_i^\ast\otimes T_0(v_i)\otimes v_k^\ast\\
&&-\rho^\ast(T_0(v_i))_{\mu}v_k^\ast\otimes v_i^\ast\otimes T_0(v_k))\mid_{\mu=1\otimes \partial \otimes 1}\\
&&-\sum_{i,k}(T_0(v_i)\otimes \rho^\ast(T_0(v_k))_{\mu}v_i^\ast\otimes v_k^\ast
-v_i^\ast\otimes [T_0(v_k)_\mu T(v_i)]\otimes v_k^\ast\\
&&-v_i^\ast\otimes \rho^\ast(T_0(v_i))_{-\mu-\partial}v_k^\ast\otimes T(v_k))\mid_{\mu=1\otimes 1\otimes \partial}\\
&&-\sum_{i,k}(-T_0(v_i)\otimes v_k^\ast\otimes \rho^\ast(T_0(v_k))_{\mu}v_i^\ast
+v_i^\ast\otimes T_0(v_k)\otimes \rho^\ast(T_0(v_i))_{-\mu-\partial}v_k^\ast\\
&&+v_i^\ast \otimes v_k^\ast \otimes [T_0(v_k)_\mu T_0(v_i)])\mid_{\mu=1\otimes \partial \otimes 1}~~~~~\text{mod} ~~(\partial^{\otimes^3}).
\end{eqnarray*}
Moreover, by Eq.~(\ref{eqn2}) and the fact that $T_0$ commutes with $\partial$, we get
\begin{eqnarray*}
&&\sum_{i,k}T_0(v_i)\otimes \rho^\ast(T_0(v_k))_{\mu}v_i^\ast\otimes v_k^\ast\mid_{\mu=1\otimes 1\otimes \partial}\\
&=&-\sum_{i,k}T_0(v_i)\otimes \sum_{j}{v_i^\ast}_{-\mu-\partial}(\rho(T_0(v_k))_\mu v_j)v_j^\ast\otimes v_k^\ast\mid_{\mu=1\otimes 1\otimes \partial}\\
\end{eqnarray*}
\begin{eqnarray*}
&=&-\sum_{i,j,k}{v_j^\ast}_\partial(\rho(T_0(v_k))_\mu(v_i))T_0(v_j)\otimes v_i^\ast\otimes v_k^\ast\mid_{\mu=1\otimes 1\otimes \partial}\\
&=&-\sum_{i,k} T_0(\sum_{j}({v_j^\ast}_\partial(\rho(T_0(v_k))_\mu(v_i)))v_j)\otimes v_i^\ast\otimes v_k^\ast\mid_{\mu=1\otimes 1\otimes \partial}\\
&=&-\sum_{i,k}T_0(\rho(T_0(v_k))_\mu(v_i))\otimes v_i^\ast\otimes v_k^\ast\mid_{\mu=1\otimes 1\otimes \partial}.
\end{eqnarray*}
Therefore by a similar study, we get
\begin{eqnarray*}
&&[[r,r]]~~~~~~\text{mod}~~(\partial^{\otimes^3})\\
&=&\sum_{i,k}((-[T_0(v_k)_\mu T_0(v_i)]+T_0(\rho(T_0(v_k))_\mu(v_i))+T_0(\rho(T_0(v_i))_{-\mu-\partial}(v_k)))\otimes v_i^\ast\otimes v_k^\ast\mid_{\mu=1\otimes 1\otimes \partial}\\
&&+(v_i^\ast\otimes (-T_0(\rho(T_0(v_k))_\mu(v_i))+[T_0(v_k)_\mu T_0(v_i)]
+T_0(\rho(T_0(v_i))_{-\mu-\partial}(v_k)))\otimes v_k^\ast)\mid_{\mu=1\otimes 1\otimes \partial}\\
&&+\sum_{i,k}(v_i^\ast\otimes v_k^\ast\otimes (-T_0(\rho(T_0(v_k))_\mu(v_i))+[T_0(v_k)_\mu T_0(v_i)]
+T_0(\rho(T_0(v_i))_{-\mu-\partial}(v_k)))\mid_{\mu=1\otimes \partial\otimes 1}\\
&=&0.
\end{eqnarray*}
Therefore $r$ is a solution of conformal CYBE if and only if $T_0$ satisfies
Eq.~(\ref{eqn3}).
\end{proof}

\begin{rmk}{\rm In fact, from the proof, when $r=r_T-r_T^{21}$ is replaced by
$$r=\sum_{i,j}b_{i,j}(-\partial^{\otimes^2},\partial\otimes 1)e_j\otimes v_i^\ast-\sum_{i,j}c_{i,j}(-\partial^{\otimes^2},1\otimes \partial) v_i^\ast \otimes e_j,$$ where $b_{i,j}(0,\partial)=c_{i,j}(0,\partial)=a_{i,j}(0,\partial)$, the conclusion still holds.}
\end{rmk}
\begin{rmk}\label{rem1}{\rm
For any $T\in \text{Chom}(V,R)$,  set $T_\lambda =
T_0+\lambda T_1+\cdots +\lambda^n T_n$ where $T_i(V)\subset R$. Suppose
$T_0$ satisfies Eq.~(\ref{eqn3}). Obviously, $T_0=0$ is an $\mathcal{O}$-operator. Then
by Theorem \ref{t2}, no matter what $T_1$, $\cdots$, $T_n$ are,
the element $r=r_T-r^{21}_T \in R\ltimes_{\rho^\ast}V^{\ast c}\otimes R\ltimes_{\rho^\ast}V^{\ast c}$ where $r_T\in
R\otimes V^{\ast c}$ corresponds to $T_\lambda$ are all
solutions of conformal CYBE in $R\ltimes_{\rho^\ast}V^{\ast c}$.

Assume that $r_T$ is given by (\ref{eqn1}) and $r=r_T-r_T^{21}$. The Lie conformal bialgebra structures are obtained through these solutions of conformal CYBE as follows (\cite{L}).
\begin{eqnarray*}
\delta(a)&=&a_\lambda r\mid_{\lambda=-\partial^{\otimes^2}}\\
&=&a_\lambda (\sum_{i,j}a_{ij}(-\partial\otimes 1-1\otimes \partial,\partial\otimes 1)e_j\otimes v_i^\ast\\
&&-\sum_{i,j}a_{ij}(-\partial\otimes 1-1\otimes \partial,1\otimes \partial)v_i^\ast \otimes e_j)\mid_{\lambda=-\partial^{\otimes^2}}\\
&=&a_\lambda (\sum_{i,j}a_{ij}(0,\partial\otimes 1)e_j\otimes v_i^\ast-\sum_{i,j}a_{ij}(0,1\otimes \partial)v_i^\ast \otimes e_j)\mid_{\lambda=-\partial^{\otimes^2}}\\
&=& a_\lambda(r_{T_0}-r_{T_0}^{21})\mid_{\lambda=-\partial^{\otimes^2}}.
\end{eqnarray*}
Therefore the solutions of conformal CYBE corresponding to $T_\lambda=\lambda T_1+\cdots +\lambda^n T_n$ do not take effect here.
Hence in the sense of Lie conformal bialgebras, the unique useful solution corresponding to $T_\lambda$ is determined by the $\mathcal{O}$-operator $T_0$.}
\end{rmk}

Finally, let us study the relation between the non-degenerate skew-symmetric solutions of conformal CYBE and 2-cocycles of Lie conformal algebras.

\begin{defi}\label{def11}{\rm
Let $R$ be a finite Lie conformal algebra which is free as a $\mathbb{C}[\partial]$-module. If $T_\lambda^r=T_0^r$ defined by Eq.~(\ref{eqn5}) is a $\mathbb{C}[\partial]$-module isomorphism from $R^{\ast c}$ to $R$, then
$r$ is called {\bf non-degenerate}. Note that in this case, Eq.~(\ref{eqn5})
becomes
\begin{eqnarray}
\{ f, T_0^r(g)\}_\lambda =\{ g\otimes f, r\}_{(-\lambda,\lambda)}~~~~~~~~~~~~\text{$f$, $g\in R^{\ast c}$.}
\end{eqnarray}}
\end{defi}

\begin{defi}{\rm Let $R$ be a Lie conformal algebra. The $\mathbb{C}$-linear map
$\alpha_\lambda: R\otimes R\rightarrow \mathbb{C}[\lambda]$ is called a {\bf 2-cocycle}
of $R$ if $\alpha_\lambda$ satisfies the following conditions
\begin{eqnarray}
\label{eqn6}\alpha_\lambda(\partial a,b)=-\lambda \alpha_\lambda(a,b),~~\alpha_\lambda(a,\partial b)=\lambda \alpha_\lambda(a,b),\\
\label{eqn7}\alpha_\lambda(a,b)=-\alpha_{-\lambda}(b,a),\\
\label{eqn8}\alpha_\lambda(a,[b_\mu c])-\alpha_\mu(b,[a_\lambda c])
=\alpha_{\lambda+\mu}([a_\lambda b],c),
\end{eqnarray}
where $a$, $b$, $c\in R$.}
\end{defi}
\begin{thm}\label{th2} Let $R$ be a Lie conformal algebra. Then
$r\in R\otimes R$ is a skew-symmetric and non-degenerate solution of conformal
CYBE in $R$ if and only if the bilinear form defined by
\begin{eqnarray}
\alpha_\lambda(a,b)=\{ (T_0^r)^{-1}(a), b\}_\lambda,~~~~~~~a,~~b\in R,
\end{eqnarray}
is a 2-cocycle on $R$, where $T_0^r\in \text{Chom}(R^{\ast c}, R)$ is the element corresponding to
$r$ through the isomorphism $R\otimes R\cong \text{Chom}(R^{\ast c}, R)$.
\end{thm}

\begin{proof}
Obviously, $\alpha_\lambda$ satisfies Eq.~(\ref{eqn6}). Since
$T_0^r$ is a $\mathbb{C}[\partial]$-module isomorphism from $R^{\ast c}$
to $R$, there exist $f$ and $g\in R^{\ast c}$ such that
$T_0^r(f)=a$ and $T_0^r(g)=b$. Therefore by the correspondence of
$r$ and $T_0^r$,
\begin{eqnarray}
\alpha_\lambda(T_0^r(f),T_0^r(g))=\{ f, T_0^r(g)\}_\lambda=\{ g\otimes f ,r \}_{(-\lambda,\lambda)}.
\end{eqnarray}
Since $r=-r^{21}$, we get
\begin{eqnarray*}
\alpha_\lambda(T_0^r(f),T_0^r(g))&=&\{ g\otimes f ,r \}_{(-\lambda,\lambda)}
=\{f \otimes g ,r^{21} \}_{(\lambda,-\lambda)}
=-\{ f \otimes g ,r \}_{(\lambda,-\lambda)}=-\{ g, T_0^r(f)\}_{-\lambda}\\
&=&-\alpha_{-\lambda}(T_0^r(g),T_0^r(f)).
\end{eqnarray*}
Therefore $\alpha_\lambda(a,b)=-\alpha_{-\lambda}(b,a)$ for any $a$, $b\in R$.
Conversely, it is also easy to see that if Eq.~(\ref{eqn7}) holds,
$r=-r^{21}$. Moreover, note that
\begin{eqnarray}\label{x1}
\{ f, T_0^r(g)\}_\lambda=-\{ g, T_0^r(f)\}_{-\lambda},~~~f,~g\in R^{\ast c}.
\end{eqnarray}

By Theorem~ \ref{tthe3},  $r$ is a solution of
conformal CYBE in $R$ if and only if $T_0^r$ is an $\mathcal{O}$-operator.
 Therefore we only need to show that Eq.~(\ref{eqn8}) is equivalent to that $T_0^r$ is an $\mathcal{O}$-operator.
Replacing $a$, $b$, $c$ by $T_0^r(f)$, $T_0^r(g)$ and $T_0^r(h)$ respectively in Eq.~(\ref{eqn8}) and according to Eq.~(\ref{x1}) and Eq.~(\ref{x2}), we get
\begin{eqnarray*}
&&\alpha_\lambda(T_0^r(f),[T_0^r(g)_\mu T_0^r(h)])-\alpha_\mu(T_0^r(g),[T_0^r(f)_\lambda T_0^r(h)])
-\alpha_{\lambda+\mu}([T_0^r(f)_\lambda T_0^r(g)],T_0^r(h))\\
&=&\{ f, [T_0^r(g)_\mu T_0^r(h)]\}_\lambda -\{ g, [T_0^r(f)_\lambda T_0^r(h)]\}_\mu+\{ h, [T_0^r(f)_\lambda T_0^r(g)]\}_{-\lambda-\mu}\\
&=&-\{ ad^\ast (T_0^r(g))_\mu f, T_0^r(h)\}_{\lambda+\mu}
+\{ad^\ast (T_0^r(f))_\lambda g, T_0^r(h)\}_{\lambda+\mu}
+\{ h, [T_0^r(f)_\lambda T_0^r(g)]\}_{-\lambda-\mu}\\
&=&\{ h, T_0^r(ad^\ast (T_0^r(g))_\mu f)\}_{-\lambda-\mu}
-\{ h, T_0^r( ad^\ast (T_0^r(f)_\lambda) g)\}_{-\lambda-\mu}
+\{h, [T_0^r(f)_\lambda T_0^r(g)]\}_{-\lambda-\mu}\\
&=& \{ h, [T_0^r(f)_\lambda T_0^r(g)]-T_0^r(ad^\ast ( T_0^r(f))_\lambda g)+T_0^r(ad^\ast (T_0^r(g))_{-\lambda-\partial} f)    \}_{-\lambda-\mu}.
\end{eqnarray*}
Since $h\in R^{\ast c}$ is arbitrary, Eq.~(\ref{eqn8}) is equivalent to that
$T_0^r$ is an $\mathcal{O}$-operator. This completes the proof.
\end{proof}

\section{Conformal CYBE and left-symmetric conformal algebras}
In this section, we investigate the relation between conformal CYBE and left-symmetric conformal algebras.

\begin{pro}\label{ll1}
There is a compatible left-symmetric conformal algebra structure on a Lie conformal algebra $R$ if and only if there exists a bijective $\mathcal{O}$-operator $T: V\rightarrow R$  associated with a certain representation $\rho$.
\end{pro}
\begin{proof}
If there is a compatible left-symmetric conformal algebra
$(R,\cdot_\lambda \cdot)$, then $R$ is an $R$-module through the left multiplication operators of
the left-symmetric conformal algebra. Then  $id: R\rightarrow R$ is a bijective
$\mathcal{O}$-operator of $R$ associated to this representation.

Conversely, suppose there exists a bijective $\mathcal{O}$-operator $T: V\rightarrow R$ of $R$ associated with a representation $\rho$. Then
\begin{eqnarray}
a_\lambda b=T(\rho(a)_\lambda T^{-1}(b)),\;\;\forall a,b\in R,
\end{eqnarray}
defines a compatible left-symmetric conformal algebra structure on $R$.
\end{proof}

\begin{thm}\label{the1}
Let $A$ be a left-symmetric conformal algebra which is free and of finite rank
as a $\mathbb{C}[\partial]$-module. Then
\begin{eqnarray}\label{eqn4}
r=\sum_{i=1}^n(e_i\otimes e_i^\ast-e_i^\ast\otimes e_i)
\end{eqnarray}
is a solution of conformal CYBE in $\mathfrak{g}(A)\ltimes_{L_A^\ast} \mathfrak{g}(A)^{\ast c}$, where
$\{e_1,\cdots,e_n\}$ is a basis of $A$ and $\{e_1^\ast,\cdots, e_n^\ast\}$ is
the dual basis of $A^{\ast c}$.
\end{thm}

\begin{proof}
By Proposition ~\ref{ll1}, $T=id:\mathfrak{g}(A)\rightarrow \mathfrak{g}(A) $ is an $\mathcal{O}$-operator associated to $L_A$.
Then by Theorem \ref{t2}, the conclusion holds.
\end{proof}

\begin{rmk}{\rm
By Theorem \ref{t2} and Remark \ref{rem1}, there are infinitely  many solutions of conformal CYBE obtained from the $\mathcal{O}$-operator $id: A\rightarrow A$.
But the $r$ given by Eq.~(\ref{eqn4}) is the unique non-degenerate solution corresponding to $id$.}
\end{rmk}
\begin{cor}
Let $A$ be a left-symmetric conformal algebra which is free and of finite rank
as a $\mathbb{C}[\partial]$-module. Then  there is a natural 2-cocycle $\alpha_\lambda$ on $\mathfrak{g}(A)\ltimes_{L_A^\ast} \mathfrak{g}(A)^{\ast c}$ given by
\begin{eqnarray}
\alpha_\lambda(a+f, b+g)=\{ f,b\}_\lambda -\{ g,a\}_{-\lambda},~~~~~~~a,~~b\in \mathfrak{g}(A),~~f,~~g\in \mathfrak{g}(A)^{\ast c}.
\end{eqnarray}
\end{cor}
\begin{proof}
It follows from Theorem \ref{th2} and Theorem \ref{the1}.
\end{proof}

\begin{pro}\label{proo1}
Let $R$ be a Lie conformal algebra and $\rho: R\rightarrow gc(V)$ be a representation of $R$. Let $T: V\rightarrow R$ be a $\mathbb{C}[\partial]$-module homomorphism.
Then
\begin{eqnarray}
u\ast_\lambda v=\rho(T(u))_\lambda v,~~~~~u,~v\in V,
\end{eqnarray}
defines a left-symmetric conformal algebra structure on $V$ if and only if
\begin{eqnarray}
[T(u)_\lambda T(v)]-T(\rho(T(u))_\lambda v-\rho(T(v))_{-\lambda-\partial}u)\in \text{ker}(\rho)[\lambda]
\end{eqnarray}
for any $u$, $v\in V$.
\end{pro}

\begin{proof}
For any $u$, $v$, $w\in V$, by a direct computation, we can get
\begin{eqnarray*}
&&(u\ast_\lambda v)_{\lambda+\mu}w-u\ast_\lambda(v\ast_\mu w)
-(v\ast_\mu u)\ast_{\lambda+\mu}w+v\ast_\mu(u\ast_\lambda w)\\
&=&-\rho([T(u)_\lambda T(v)]-T(\rho(T(u))_\lambda v-\rho(T(v)_{-\lambda-\partial} u)))_{\lambda+\mu}w.
\end{eqnarray*}
Then the conclusion holds.
\end{proof}

\begin{cor}\label{cor1}
Let $R$ be a Lie conformal algebra and $\rho: R\rightarrow gc(V)$ be a representation of $R$. Suppose $T: V\rightarrow R$ is an $\mathcal{O}$-operator associated to $\rho$.
Then  the following $\lambda$-product
\begin{eqnarray}
u\ast_\lambda v=\rho(T(u))_\lambda v,~~~~~~\text{$u$, $v\in V$,}
\end{eqnarray}
endows a left-symmetric conformal algebra structure on $V$. Therefore
$V$ is a Lie conformal algebra which is the sub-adjacent Lie conformal algebra of
this left-symmetric conformal algebra, and $T:V\rightarrow R$ is a homomorphism of Lie conformal algebra. Moreover, $T(V)\subset R$ is a Lie conformal subalgebra of $R$
and there is also a natural left-symmetric conformal algebra structure on $T(V)$ defined as follows
\begin{eqnarray}\label{ww1}
T(u)_\lambda T(v)=T(u\ast_\lambda v)=T(\rho(T(u))_\lambda v),~~~\text{$u$, $v\in V$.}
\end{eqnarray}
In addition, the sub-adjacent Lie conformal algebra of this left-symmetric conformal algebra is a subalgebra of $R$ and $T: V\rightarrow R$ is a homomorphism of left-symmetric conformal algebra.
\end{cor}
\begin{proof}
It can directly obtained from Proposition \ref{proo1}.
\end{proof}

\begin{cor}\label{x3}
Let $R$ be a Lie conformal algebra and $T: R\rightarrow R$ is a Rota-Baxter operator of weight zero. Then there is a left-symmetric conformal algebra structure on $R$ with the following $\lambda$-product
\begin{eqnarray}
a_\lambda b=[T(a)_\lambda b],~~~~~a,~b\in R.
\end{eqnarray}
\end{cor}
\begin{proof}
It follows directly from Corollary \ref{cor1}.
\end{proof}
\section{$\mathcal{O}$-operator and conformal $S$-equation}

\begin{thm}\label{the3}
Let $A$ be a finite left-symmetric conformal algebra which is free as a $\mathbb{C}[\partial]$-module and $r\in A\otimes A$ be symmetric. Then $r$ is a solution of
conformal $S$-equation if and only if the $T\in \text{Chom}(A^{\ast c},A)$ corresponding to $r$ satisfies
\begin{eqnarray}
[T_0(a^\ast)_\lambda T_0(b^\ast)]= T_0(L_A^{\ast}(T_0(a^\ast))_\lambda b^\ast
-L_A^{\ast}(T_0(b^\ast))_{-\lambda-\partial} a^\ast),\;\;\forall a^\ast, b^\ast\in A^{\ast c},
\end{eqnarray}
where $T_0=T_\lambda\mid_{\lambda=0}$.

\end{thm}
\begin{proof}
It follows from a similar proof as the one in Theorem \ref{tthe3}.
\end{proof}


\begin{cor}
Let $A$ be a finite left-symmetric conformal algebra which is free as a $\mathbb{C}[\partial]$-module. Suppose $r$ is a symmetric solution of
conformal $S$-equation. Then $T\in \text{Chom}(A^{\ast c},A)$ corresponding to
$r$ is an $\mathcal{O}$-operator associated to $L_A^\ast$. Therefore there is
a left-symmetric conformal algebra structure on $A^{\ast c}$ given as follows
\begin{eqnarray}
a^\ast \circ_\lambda b^\ast=L_A^\ast(T(a^\ast))_\lambda b^\ast, ~~~~a^\ast,~~b^\ast\in A^{\ast c}.
\end{eqnarray}

\end{cor}
\begin{proof}
It follows from Corollary \ref{cor1} and Theorem \ref{the3}.
\end{proof}

\begin{thm}\label{the5}
Let $R$ be a finite Lie conformal algebra which is free as a $\mathbb{C}[\partial]$-module and $V$ a free and finitely generated $\mathbb{C}[\partial]$-module.
Set $\{e_i\}\mid_{i=1,\cdots,n}$ and $\{v_i\}\mid_{i=1,\cdots,m}$ be the
$\mathbb{C}[\partial]$-basis of $R$ and $V$ respectively. Suppose
$\rho:R\rightarrow gc(V)$ is a representation of $R$ and
$\rho^\ast: R\rightarrow gc(V^{\ast c})$ is its dual representation.
Let $T\in \text{Chom}(V,R)$ with $T_0=T_\lambda\mid_{\lambda=0}: V\rightarrow R$ be
an $\mathcal{O}$-operator associated to $\rho$ where
$T_\lambda(v_i)=\sum_{j=1}^na_{ij}(\lambda,\partial)e_j$ for any $i\in \{1,\cdots,m\}$.
Then
\begin{eqnarray*}
r=\sum_{i,j}(a_{ij}(-1\otimes \partial-\partial\otimes 1,\partial\otimes 1)e_j\otimes v_i^\ast+a_{ij}(-1\otimes \partial-\partial\otimes 1,1\otimes \partial)v_i^\ast \otimes e_j),
\end{eqnarray*}
is a symmetric solution of conformal $S$-equation in $T_0(V)\ltimes_{\rho^\ast,0} V^{\ast c}$ where the left-symmetric conformal algebra structure on $T_0(V)$ is given by
Eq.~(\ref{ww1}).
\end{thm}
\begin{proof}
Obviously, $r$ is symmetric. Then  we only need to show that
$r$ satisfies the conformal $S$-equation in $T_0(V)\ltimes_{\rho^\ast,0} V^{\ast c}$.
Note that $T_0(u)_\lambda T_0(v)=T_0(\rho(T_0(u))_\lambda v)$.
Then  with a similar proof as of Theorem \ref{t2}, we have
\begin{eqnarray*}
&&\{\{r,r\}\}~~~\text{mod} ~~(\partial^{\otimes^3})\\
&=&\sum_{i,k}(T_0(v_k)_\mu T_0(v_i)\otimes v_k^\ast \otimes v_i^\ast
+\rho^\ast(T_0(v_k))_\mu v_i^\ast \otimes v_k^\ast \otimes T_0(v_i))\mid_{\mu=1\otimes \partial\otimes 1}\\
&&-\sum_{i,k}(v_k^\ast\otimes T_0(v_k)_\mu T_0(v_i)\otimes v_i^\ast
+v_k^\ast\otimes \rho^\ast(T_0(v_k))_\mu(v_i^\ast)\otimes T_0(v_i))\mid_{\mu=\partial \otimes 1\otimes 1}\\
&&+\sum_{i,k}(T_0(v_i)\otimes v_k^\ast\otimes \rho^\ast(T_0(v_k))_{-\mu-\partial} v_i^\ast
-v_i^\ast\otimes T_0(v_k)\otimes \rho^\ast(T_0(v_i))_\mu v_k^\ast\\
&&-v_i^\ast\otimes v_k^\ast\otimes [T_0(v_i)_\mu T_0(v_k)])\mid_{\mu=\partial \otimes 1\otimes 1} ~~~\text{mod} ~~(\partial^{\otimes^3})\\
&=&\sum_{i,k}(T_0(\rho(T_0(v_k))_\mu v_i)\otimes v_k^\ast\otimes v_i
-v_i^\ast\otimes v_k^\ast\otimes T_0(\rho(T_0(v_k))_\mu v_i))\mid_{\mu=1\otimes \partial\otimes 1}\\
&&-\sum_{i,k}(v_k^\ast\otimes T_0(\rho(T(v_k))_\mu v_i)\otimes v_i^\ast
-v_k^\ast\otimes v_i^\ast\otimes T_0(\rho(T_0(v_k))_\mu(v_i)))\mid_{\mu=\partial\otimes 1\otimes 1}\\
&&-\sum_{i,k}T_0(\rho(T_0(v_k))_\mu v_i)\otimes v_k^\ast\otimes v_i\mid_{\mu=1\otimes \partial\otimes 1}\\
&&+\sum_{i,k}(v_k^\ast\otimes T_0(\rho(T_0(v_k))_\mu v_i)\otimes v_i^\ast-v_k^\ast\otimes v_i^\ast\otimes T_0(\rho(T_0(v_k))_\mu v_i))\mid_{\mu=\partial\otimes 1\otimes 1}\\
&&+\sum_{i,k}v_i^\ast\otimes v_k^\ast \otimes T_0(\rho(T_0(v_k))_\mu v_i\mid_{\mu=1\otimes \partial\otimes 1}=0.
\end{eqnarray*}
This completes the proof.
\end{proof}


\begin{cor}\label{cv}
Let $A$ be a finite left-symmetric conformal algebra which is free
as a $\mathbb{C}[\partial]$-module. Then
\begin{eqnarray}
r=\sum_{i=1}^n(e_i\otimes e_i^\ast+e^\ast\otimes e_i)
\end{eqnarray}
is a non-degenerate symmetric solution of conformal $S$-equation in $A \ltimes_{L_A^\ast,0} A^{\ast c}$
where $\{e_i\}\mid_{i,\cdots,n}$ is a $\mathbb{C}[\partial]$-basis of
$A$ and $\{e_i^\ast\}\mid_{i,\cdots,n}$ is its dual basis of $A^{\ast c}$.
\end{cor}
\begin{proof}
It follows from Theorem \ref{the5} by letting $V=A$, $\rho=L_A$
and $T=id_A$.
\end{proof}


\begin{defi}{\rm
Let $A$ be a left-symmetric conformal algebra. A bilinear form $\beta_\lambda:A\otimes A\rightarrow \mathbb{C}[\lambda]$ is called a {\bf 2-cocycle} of $A$ if
$\beta_\lambda$ satisfies the following properties:
\begin{eqnarray}
\label{ww9}\beta_\lambda(\partial a, b)=-\lambda\beta_\lambda(a, b),~~~\beta_\lambda( a, \partial b)=\lambda\beta_\lambda(a, b),\\
\label{ww2}\beta_{\lambda+\mu}(a_\lambda b,c)-\beta_\lambda(a,b_\mu c)
=\beta_{\lambda+\mu}(b_\mu a,c)-\beta_\mu(b,a_\lambda c),
\end{eqnarray}
for any $a$, $b$, $c\in A$. A 2-cocycle $\beta_\lambda: A\otimes A\rightarrow \mathbb{C}[\lambda]$ of $A$ is called {\bf symmetric} if
\begin{eqnarray}
\beta_\lambda(a,b)=\beta_{-\lambda}(b,a),~~~~~~~~~~~~\text{for any $a$, $b\in A$.}
\end{eqnarray}}
\end{defi}

\begin{rmk}{\rm
It is shown in \cite{HL} that a 2-cocycle of $A$ is equivalent to a central extension
of $A$ by the trivial $\mathbb{C}[\partial]$-module $\mathbb{C}$.}
\end{rmk}

\begin{thm}\label{th9}
Let $A$ be a finite left-symmetric conformal algebra which is free as a $\mathbb{C}[\partial]$-module.
Then $r$ is a symmetric non-degenerate solution of conformal $S$-equation if and only if the bilinear form
defined by
\begin{eqnarray}
\label{ww3}\beta_\lambda(a,b)=\{ T^{-1}(a), b\}_\lambda,~~~~~~~~a,~~b\in A,
\end{eqnarray}
is a symmetric  2-cocycle of $A$ where $T:A^{\ast c}\rightarrow A$ is the $\mathbb{C}[\partial]$-module homomorphism
corresponding to $r$ by the $\mathfrak{g}(A)$-module isomorphism $A\otimes A\cong \text{Chom}(A^{\ast c},A)$. Here, $A$ is seen as a module
of $\mathfrak{g}(A)$ via the left multiplication operators of $A$.
\end{thm}

\begin{proof}
 By Eq.~(\ref{ww3}), Eq.~(\ref{ww9}) is satisfied.
Suppose that $r$ is symmetric. Then
\begin{eqnarray}
\label{ww4}\{f, T(g)\}_\lambda=\{ g\otimes f, r\}_{(-\lambda,\lambda)}
=\{f\otimes g, r\}_{(\lambda,-\lambda)}
=\{ g, T(f)\}_{-\lambda},\;\;\forall f, g \in A^{\ast c}.
\end{eqnarray}
Since $T$ is a $\mathbb{C}[\partial]$-module isomorphism from
$A^{\ast c}\rightarrow A$, there exist $f$ and $g$ such that
$T(f)=a$ and $T(g)=b$. By Eq.~(\ref{ww4}), we get
$$\beta_\lambda(a,b)=\{ f, T(g)\}_\lambda=\{ g, T(f)\}_{-\lambda}
=\beta_{-\lambda}(b,a).$$ Therefore $\beta_\lambda$ is symmetric.
With a similar process, it is easy to see that if $\beta_\lambda$ is symmetric,
then $r$ is symmetric. Therefore $r$ is symmetric if and only if
$\beta_\lambda$ is symmetric.

Set $r=\sum_ir_i\otimes l_i\in A\otimes A$. Since $r$ is symmetric, $r=\sum_i r_i\otimes l_i=\sum_i l_i\otimes r_i$.
Since $\{ f, T(g)\}_\lambda=\{ g\otimes f, r\}_{(-\lambda,\lambda)}$, we get $T(g)=\sum_ig_{-\partial}(r_i)l_i=\sum_ig_{-\partial}(l_i)r_i$ for
any $g\in A^{\ast c}$. Since $T$ is invertible, there exist $f$, $g$, $h\in A^{\ast c}$ such that $T(f)=a$, $T(g)=b$ and $T(h)=c$.
Then
\begin{eqnarray*}
&&\beta_{\lambda+\mu}(a_\lambda b,c)-\beta_\lambda(a,b_\mu c)
-\beta_{\lambda+\mu}(b_\mu a,c)+\beta_\mu(b,a_\lambda c)\\
&=&\{ h, T(f)_\lambda T(g)\}_{-\lambda-\mu}
-\{ f,T(g)_\mu T(h)\}_\lambda
-\{ h, T(g)_\mu T(f)\}_{-\lambda-\mu}
+\{g, T(f)_\lambda T(h)\}_\mu\\
&=&\{ f\otimes g\otimes h,\sum_{ij} r_i\otimes r_j
\otimes {l_i}_\lambda l_j\}_{(\lambda, \mu,-\lambda-\mu)}
-\{f\otimes g\otimes h,\sum_{ij} {r_i}_\mu r_j
\otimes {l_i}\otimes l_j\}_{(\lambda, \mu,-\lambda-\mu)}\\
&&-\{ f\otimes g\otimes h,\sum_{ij} l_j\otimes l_i
\otimes {r_i}_\mu r_j\}_{(\lambda, \mu,-\lambda-\mu)}
+\{ f\otimes g\otimes h,\sum_{ij} l_i\otimes {r_i}_\lambda r_j
\otimes l_j\}_{(\lambda, \mu,-\lambda-\mu)}\\
&=& \{f\otimes g\otimes h, \sum_{ij} r_i\otimes r_j
\otimes [{l_i}_{\mu^{'}} l_j]\mid_{\mu^{'}=\partial\otimes 1\otimes 1}\\
&&-({l_j}_{\mu^{'}} r_i\otimes r_i\otimes l_i)\mid_{\mu^{'}=1\otimes \partial\otimes 1}
+(r_j\otimes {l_i}_{\mu^{'}} r_i\otimes l_i)\mid_{\mu^{'}=\partial\otimes 1\otimes 1}\}_{(\lambda, \mu,-\lambda-\mu)}.
\end{eqnarray*}
Therefore $\beta_\lambda$ is a 2-cocycle of $A$ if and only if
$r=0~~\text{mod} (\partial^{\otimes^3})$. This competes the proof.
\end{proof}

\begin{cor}
Let $A$ be a finite left-symmetric conformal algebra which is free as a
$\mathbb{C}[\partial]$-module. Then there is a natural 2-cocycle
$\beta_\lambda$ of $A\ltimes_{L^{\ast},0} A^{\ast c}$ given by
\begin{eqnarray}
\beta_\lambda(a+f,b+g)=\{ f,b\}_\lambda+\{ g, b\}_{-\lambda},~~~a,~b\in A,~~f,~g\in A^{\ast c}.
\end{eqnarray}

\end{cor}
\begin{proof}
It can be directly obtained from Corollary \ref{cv} and Theorem \ref{th9}.
\end{proof}

\section{Rota-Baxter operators on Lie conformal algebras}

It is known that Rota-Baxter operators (of weight zero) on Lie conformal algebras are a class of
$\mathcal{O}$-operators and hence they can be used to provide the solutions of conformal CYBE and $S$-equation in the previous sections.
In this section, we give a further study of Rota-Baxter operators with any weight on Lie conformal algebras.

Recall the definition of Rota-Baxter operator of weight $\alpha$ on an associative (or a non-associative) algebra.
\begin{defi}{\rm
Let $A$ be an associative (or a nonassociative) algebra with the operation
$\cdot: A\otimes A\rightarrow A$ over $\mathbb{C}$. A linear operator
$P:A\rightarrow A$ is called a {\bf Rota-Baxter operator of weight $\alpha$} $(\alpha \in \mathbb{C})$ on
$A$ if $P$ satisfies the following condition:
\begin{eqnarray}
P(x)\cdot P(y)=P(P(x)\cdot y+x\cdot P(y))+\alpha P(x\cdot y),~~\forall~~\text{$x$,~~$y\in A.$}
\end{eqnarray}
If $\alpha=0$, we call $P$ a {\bf Rota-Baxter operator} simply.}

\end{defi}

\begin{rmk}{\rm
In particular, let $R$ be a Lie conformal algebra. If $T:R\rightarrow R$ is a $\mathbb{C}[\partial]$-module homomorphism satisfying
\begin{eqnarray}\label{ooo3}
[T(a)_\lambda T(b)]=T([a_\lambda T(b)])+T([T(a)_\lambda b])+\alpha T[a_\lambda b],~~~\forall~~\text{$a$, $b\in R$,  }
\end{eqnarray}
then $T$ is called a {\bf Rota-Baxter operator of weight $\alpha$} on $R$, where $\alpha \in \mathbb{C}$.
In fact, the case of weight zero has already been studied in the previous sections (for example, see Definition~\ref{RB0}).
}
\end{rmk}
\begin{pro}\label{p1}
Let $R$ be a Lie conformal algebra and $T$ be a Rota-Baxter operator of weight $\alpha$ on
$R$. Define the linear operator $\mathcal{T}: \text{Coeff}(R)\rightarrow \text{Coeff}(R)$ as follows.
\begin{eqnarray}
\mathcal{T}(a_n)=T(a)_n.
\end{eqnarray}
Then $\mathcal{T}$ is a Rota-Baxter operator of weight $\alpha$ on the Lie algebra $\text{Coeff}(R)$.
\end{pro}
\begin{proof}
It is directly  obtained from  Eqs. (\ref{oo3}) and (\ref{eq1}).
\end{proof}

\begin{rmk}{\rm
Note that a similar study about averaging operators on Lie conformal algebras
has been given in \cite{Ko}.}
\end{rmk}

\begin{defi}{\rm
Let $R=\mathbb{C}[\partial]V$ be a free $\mathbb{C}[\partial]$-module. Then by the definition of
$\mathcal{T}$ in Proposition \ref{p1}, $\mathcal{T}: \mathbb{C} a_n\rightarrow \mathbb{C} a_n $ if and only if $T(a)=f(a) a$ for any $a\in V$ where $f(a)\in \mathbb{C}$.
If $T: R\rightarrow R$ satisfies this condition, $T$ is called {\bf homogeneous}.}
\end{defi}

Next we study Rota-Baxter operators on a class of
Lie conformal algebras named quadratic Lie conformal algebras (\cite{X1}).

\begin{defi}{\rm
 If $R=\mathbb{C}[\partial]V$ is a Lie conformal algebra as a free
$\mathbb{C}[\partial]$-module and the $\lambda$-bracket is of the following form:
\begin{eqnarray*}
[a_{\lambda} b]=\partial u+\lambda v+ w,~~~~~~\text{$a$, $b\in V$,}
\end{eqnarray*}
where $u$, $v$, $w\in V$, then $R$ is called a {\bf quadratic Lie conformal algebra}.}
\end{defi}

\begin{defi}{\rm
A {\bf Novikov algebra} $(A, \circ)$ is a vector space $A$ with an operation
``$\circ$" satisfying the following axioms: for $a$, $b$, $c\in A$,
\begin{eqnarray*}
&(a\circ b)\circ c=(a\circ c)\circ b,\\
&(a\circ b)\circ c-a\circ (b\circ c)=(b\circ a)\circ c-b\circ(a\circ c).
\end{eqnarray*}}
\end{defi}
\begin{defi}{\rm (\cite{X1})
A {\bf Gel'fand-Dorfman bialgebra} $(A, [\cdot, \cdot], \circ)$ is a vector space $A$ with two algebraic operations
 $[\cdot,\cdot]$ and $\circ$ such that $(A,[\cdot,\cdot])$ forms a Lie algebra, $(A,\circ)$ forms a Novikov algebra and the following compatibility condition holds:
\begin{eqnarray}\label{xx1}
[a\circ b, c]+[a,b]\circ c-a\circ [b,c]-[a\circ c, b]-[a,c]\circ b=0,
\end{eqnarray}
for $a$, $b$, and $c\in A$. If a linear operator $P:A\rightarrow A$ is a Rota-Baxter operator of weight $\alpha$ on $(A,[\cdot,\cdot])$ as well as $(A,\circ)$,
then $P$ is called a {\bf Rota-Baxter operator of weight $\alpha$ } on the Gel'fand-Dorfman bialgebra $(A, [\cdot, \cdot], \circ)$.}
\end{defi}

\begin{thm}{\rm (\cite{X1} or \cite{GD})}
A  Lie conformal algebra $R=\mathbb{C}[\partial]V$ is quadratic if and only if $V$ is a  Gel'fand-Dorfman bialgebra. The correspondence is given as follows.
\begin{eqnarray}
[a_\lambda b]=\partial(b\circ a)+\lambda(a\ast b)+[b,a], ~~~\text{for any $a$. $b\in V$.}
\end{eqnarray}
Here $a\ast b=a\circ b+b\circ a$ and $(V,[\cdot,\cdot],\circ)$ is a Gel'fand-Dorfman bialgebra.
\end{thm}

Then we study the Rota-Baxter operators of weight $\alpha$ on quadratic Lie conformal algebras.

\begin{thm}\label{tt1}
Let $R=\mathbb{C}[\partial]V$ be the finite quadratic Lie conformal algebra corresponding
to a Gel'fand-Dorfman bialgebra $(V,[\cdot,\cdot],\circ)$.
If the algebra $(V,\ast)$ has no zero divisors, then any Rota-Baxter operator $T$  of weight $\alpha$ on $R$ is homogenous and $T$ is just the Rota-Baxter operator of weight $\alpha$  on the Gel'fand-Dorfman bialgebra $(V,[\cdot,\cdot],\circ)$.
\end{thm}

\begin{proof}
Since $R$ is free and finite as a $\mathbb{C}[\partial]$-module, $V$ is a finite-dimensional vector space. Set
$$T=\sum_{i=0}^n\partial^iT_i,$$
where $T_i: V\rightarrow V$ $(i=0, \cdots, n)$ are linear maps  and there exists some $a\in V$ such that
$T_n(a)\neq 0$.
Therefore
\begin{eqnarray*}
[T(a)_\lambda T(b)]
&=&[\sum_{i=0}^n\partial^iT_i(a)_\lambda \sum_{j=0}^n\partial^jT_j(b)]\\
&=&\sum_{i,j=0}^n(-\lambda)^i(\lambda+\partial)^j(\partial(T_j(b)\circ T_i(a))
+\lambda(T_j(b)\ast T_i(a))+[T_j(b),T_i(a)]),
\end{eqnarray*}
and
\begin{eqnarray*}
&&T([T(a)_\lambda b]+[a_\lambda T(b)])+\alpha T[a_\lambda b]\\
&&=\sum_{i=0}^n(-\lambda)^i(\partial T(b\circ T_i(a))+\lambda T(T_i(a)\ast b)
+T([b, T_i(a)]))\\
&&+\sum_{j=0}^n(\lambda+\partial)^j(\partial T(T_j(b)\circ a)
+\lambda T(T_j(b)\ast a)+T([T_j(b),a])\\
&&+\alpha(\partial T(b\circ a)+\lambda T(b\ast a)+T([b,a])).
\end{eqnarray*}

If $n\geq 1$, by comparing the coefficients of $\lambda^{2n+1}$,
we get $T_n(b)\ast T_n(a)=0$ for any $a$, $b\in V$. By the assumption,
there exists $a\in V$ such that $T_n(a)\neq 0$. Hence
$T_n(a)\ast T_n(a)=0$, which contradicts with the assumption that $(V,\ast)$ has no zero divisors. Therefore in this case, $T(a)=T_0(a)$ for any $a\in V$.  Hence
$T_0$ is exactly the Rota-Baxter operator $T$ of weight $\alpha$ on the Gel'fand-Dorfman bialgebra $(V,[\cdot,\cdot],\circ )$.
This completes the proof.
\end{proof}

\begin{cor}\label{y1}
Let $T$ be a Rota-Baxter operator of weight $\alpha$  on the a Gel'fand-Dorfman bialgebra $(V,[\cdot,\cdot],\circ )$. Then
$\widetilde{T}: R=\mathbb{C}[\partial]V\rightarrow  \mathbb{C}[\partial]V$ given by $\widetilde{T}(a)=T(a)$ for any $a\in V$  is a Rota-Baxter operator of weight $\alpha$
on $R$.
\end{cor}
\begin{proof}
It can be directly obtained from Theorem \ref{tt1}.
\end{proof}
\begin{cor}\label{y2}
Let $T$ be a Rota-Baxter operator $T$ of weight $\alpha$ on a Gel'fand-Dorfman bialgebra $(V,[\cdot,\cdot],\circ )$. Let $R=\mathbb{C}[\partial]V$ be the
corresponding quadratic Lie conformal algebra. Then
$\mathcal{T}: \text{Coeff}(R)\rightarrow \text{Coeff}(R)$ given by
$ \mathcal{T}(a_n)=T(a)_n$ where $a\in V$ and $n\in \mathbb{Z}$ is
a Rota-Baxter operator $T$ of weight $\alpha$.
\end{cor}

\begin{proof}
It can be obtained by Corollary \ref{y1} and Proposition \ref{p1}.
\end{proof}


Finally two examples about Rota-Baxter operators of weight $0$ are given.

\begin{ex}{\rm
The Virasoro Lie conformal algebra $\text{Vir}$ is the simplest nontrivial
example of Lie conformal algebras. It is defined by
$$\text{Vir}=\mathbb{C}[\partial]L, ~~[L_\lambda L]=(\partial+2\lambda)L.$$
Coeff$\text{(Vir)}$ is just the Witt algebra.

Now we consider the Rota-Baxter operators on $\text{Vir}$.
In fact, $\text{Vir}$ is a quadratic Lie conformal algebra
corresponding to the Novikov algebra $(V=\mathbb{C}L, \circ)$
where $L\circ L=L$. Obviously, $(V,\ast)$ has no zero divisors. Therefore
by Theorem \ref{tt1}, we only need to consider Rota-Baxter operators on $V$. It is easy to see that
any Rota-Baxter operator on $V$ is of the form
$T(L)=0$. Therefore all Rota-Baxter operators on $\text{Vir}$ are trivial.}
\end{ex}

\begin{ex}{\rm
Let $R=\mathbb{C}[\partial]L\oplus \mathbb{C}[\partial]W$ be a Lie conformal algebra of rank 2 with the $\lambda$-bracket given by
\begin{eqnarray}
[L_\lambda L]=(\partial+2\lambda)L,~~[L_\lambda W]=(\partial+\lambda)W,~~
[W_\lambda W]=0.
\end{eqnarray}
$\text{Coeff}(R)$ is isomorphic to the Heisenberg-Virasoro Lie algebra which is
spanned by $\{L_n, W_n\mid n\in \mathbb{Z}\}$ satisfying
\begin{eqnarray}
[L_m, L_n]=(m-n)L_{m+n},~~~[L_m,W_n]=-nW_{m+n},~~[W_m,W_n]=0.
\end{eqnarray}
Obviously, the Gel'fand-Dorfman bialgebra corresponding to the Heisenberg-Virasoro Lie
conformal algebra does not satisfy the condition in Theorem \ref{tt1}.
By a direct computation, we show that any Rota-Baxter operator on
$R$ is one of the following two forms:
\begin{eqnarray}
T(L)=-b(L+W),~~T(W)=b(L+W),~~~\text{where $b\in \mathbb{C}\backslash\{0\}$;}
\end{eqnarray}
and
\begin{eqnarray}
T(L)=g(\partial)W,~~T(W)=0,~~~\text{where $g(\partial)\in \mathbb{C}[\partial]$.}
\end{eqnarray}
By Proposition \ref{p1}, we obtain two classes of Rota-Baxter operators
on the Heisenberg-Virasoro Lie algebra as follows.
One is
\begin{eqnarray*}
T(L_m)=-b(L_m+W_m),~~T(W_m)=b(L_m+W_m)~~~~\text{where $b\in \mathbb{C}\backslash\{0\}$};
\end{eqnarray*}
another is
\begin{eqnarray*}
T(L_m)=a,~~T(W_m)=0,~~~\text{where $a\in \mathbb{C}\{W_m\mid m\in \mathbb{Z}\}$.}
\end{eqnarray*}

In addition, by Corollary ~\ref{x3}, from the Rota-Baxter operators
on the Heisenberg-Virasoro Lie conformal algebra, we can naturally get the following  two left-symmetric conformal algebras structures on $A=\mathbb{C}[\partial]L\oplus \mathbb{C}[\partial]W$:
one is
\begin{eqnarray*}
L_\lambda L=-b((\partial+2\lambda)L+\lambda W),  ~~~~L_\lambda W=-b(\partial+\lambda) W,\\
W_\lambda L=b((\partial+2\lambda)L+\lambda W),~~~~W_\lambda W=b(\partial+\lambda) W;
\end{eqnarray*}
another is
\begin{eqnarray*}
L_\lambda L=g(-\lambda)\lambda W,  ~~~~L_\lambda W=W_\lambda L=W_\lambda W=0.
\end{eqnarray*}
}

\end{ex}

{\bf Acknowledgments}
{
This work was  supported by the National Natural Science Foundation of China (11425104, 11501515),
the Zhejiang Provincial Natural Science Foundation of China (LQ16A010011) and the Scientific Research Foundation of Zhejiang Agriculture
and Forestry University (2013FR081). This work was carried out during the first
author's stay at Chern Institute of Mathematics, Tianjin, China, from April 10st to
April 24th, 2016, and he would like to thank the CIM for its support and hospitality.}

\end{document}